\documentclass[a4paper,reqno,12pt]{amsart}
\usepackage[T1]{fontenc}
\usepackage[utf8]{inputenc}
\usepackage[english]{babel}

\usepackage{amsmath}
\usepackage{amsthm}
\usepackage{amssymb}
\usepackage{mathtools}
\usepackage{amsmath}
\usepackage{amsfonts}
\usepackage{mathrsfs}
\usepackage{esint}
\usepackage{yfonts}
\usepackage{quoting}
\usepackage{graphics}
\usepackage{booktabs}
\usepackage{bm}
\usepackage{bbm}
\usepackage{lmodern}
\usepackage{stmaryrd}
\usepackage{caption}
\usepackage{enumitem}			
\usepackage{tikz}
\usepackage{bookmark}
\usepackage{pgfplots}
\usepgfplotslibrary{fillbetween}
\usepackage{orcidlink}
\usepackage[mathscr]{euscript}
\usepackage{caption}
\usepackage{subcaption}
\usepackage[]{old-arrows}		

\newcommand{\sH}{{\!\!\!\mbox{\fontsize{6}{0} \selectfont $H$}}}

\newcommand{\sGt}{{\!\!\!\mbox{\fontsize{6,5}{0} \selectfont $\Gamma_t$}}}

\usepackage{hyperref}
\hypersetup{
	colorlinks=true,
	linkcolor=blue,
	citecolor=blue,
	urlcolor=black,
	linktoc=all
}

\usepackage[margin=1.1in]{geometry}

\quotingsetup{font=small}

\theoremstyle{plain}
\newtheorem{lemma}{Lemma}[section]

\theoremstyle{plain}
\newtheorem{proposition}[lemma]{Proposition}

\theoremstyle{plain}
\newtheorem{theorem}[lemma]{Theorem}
\numberwithin{equation}{section}	

\theoremstyle{plain}

\theoremstyle{definition}
\newtheorem{remark}{Remark}[section]

\DeclarePairedDelimiter{\abs}{\lvert}{\rvert}
\DeclarePairedDelimiter{\norma}{\lVert}{\rVert}

\renewcommand{\epsilon}{\varepsilon}

\DeclareMathOperator{\dist}{dist}

\DeclareMathOperator{\cH}{\mathcal H}

\newcommand{\numberset}{\mathbb}
\newcommand{\N}{\numberset{N}}			
\newcommand{\Z}{\numberset{Z}}			
\newcommand{\R}{\numberset{R}}			
\newcommand{\sfera}{\numberset{S}}		
\newcommand{\Haus}{\mathcal{H}}			
			
\newcommand{\loc}{{\rm loc}}

\def\XXint#1#2#3{{\setbox0=\hbox{$#1{#2#3}{\int}$ }
		\vcenter{\hbox{$#2#3$ }}\kern-.6\wd0}}

\usepackage{enumitem}

\allowdisplaybreaks

\begin{document}
	
	\title[Fractional Dirichlet problems with overdetermined Neumann condition]{Fractional Dirichlet problems with an overdetermined nonlocal Neumann condition}
	
	\author[Michele Gatti]{Michele Gatti \orcidlink{0009-0002-6686-9684}} 
	\address[]{Michele Gatti. Dipartimento di Matematica ‘Federigo Enriques’, Università degli Studi di Milano, Via Cesare Saldini 50, 20133, Milan, Italy}
	\email{michele.gatti1@unimi.it}
	
	\author[Julian Scheuer]{Julian Scheuer}
	\address[]{Julian Scheuer. Institut f\"ur Mathematik, Goethe-Universit\"at Frankfurt, Robert-Mayer-Stra\ss{}e 10, 60629, Frankfurt am Main, Germany}
	\email{scheuer@math.uni-frankfurt.de}
	
	\author[Tobias Weth]{Tobias Weth}
	\address[]{Tobias Weth. Institut f\"ur Mathematik, Goethe-Universit\"at Frankfurt, Robert-Mayer-Stra\ss{}e 10, 60629, Frankfurt am Main, Germany}
	\email{weth@math.uni-frankfurt.de}
	
	\date{\today}
	
	\begin{abstract}
		We investigate symmetry and quantitative approximate symmetry for an overdetermined problem related to the fractional torsion equation in a regular open, bounded set~$\Omega \subseteq \R^n$. Specifically, we show that if~$\overline{\Omega}$ has positive reach and the nonlocal normal derivative introduced in~\cite{drov} is constant on an external surface parallel and sufficiently close to~$\partial \Omega$, then~$\Omega$ must be a ball. Remarkably, this conclusion remains valid under the sole assumption that~$\Omega$ is convex.
		
		Moreover, we analyze the quantitative stability of this result under two distinct sets of assumptions on~$\Omega$.
		
		Finally, we extend our analysis to a broader class of overdetermined Dirichlet problems involving the fractional Laplacian.
	\end{abstract}
	
	\maketitle
	
	
	\section{Introduction}
	\label{sec:intro-fractional}
	
	Let~$\Omega \subseteq \R^n$ be an open, bounded set with outer parallel sets
	\begin{equation*}
		\Gamma_t \coloneqq \{x \in \R^n \mid \dist(x,\Omega) = t\} \quad\text{with } t>0,
	\end{equation*} 
	and consider the problem
	\begin{equation}
		\label{eq:mainprob-frac}
		\begin{cases}
			\begin{aligned}
				& \!\left(-\Delta\right)^s u = 1	&& \text{in } \Omega, \\
				& u = 0 							&& \text{in } \R^n \setminus \Omega,
			\end{aligned}
		\end{cases}
	\end{equation}
	with the overdetermined Neumann condition
	\begin{equation}
		\label{eq:over-cond}
		\mathcal{N}_s u = c_\flat \in \R \quad\text{on } \Gamma_t,
	\end{equation}
	for some~$t>0$.	Here, for~$s \in (0,1)$,~$\left(-\Delta\right)^s$ is the fractional Laplacian defined for~$u \in C^\infty_c(\R^n)$ as
	\begin{equation*}
		\left(-\Delta\right)^s u(x) \coloneqq c_{n,s} \,\mathrm{P.V.} \int_{\R^n} \frac{u(x)-u(y)}{\abs*{x-y}^{n+2s}} \, dy,
	\end{equation*}
	with
	\begin{equation*}
		c_{n,s} \coloneqq s \left(1-s\right) 4s \,\pi^{-n/2} \,\frac{\Gamma(n/2+s)}{\Gamma(2-s)},
	\end{equation*}
	and~$\mathcal{N}_s$ is the nonlocal normal derivative given by
	\begin{equation*}
		\label{eq:nonloc-norm-der-intro}
		\mathcal{N}_s u(x) \coloneqq c_{n,s} \int_{\Omega} \frac{u(x)-u(y)}{\abs*{x-y}^{n+2s}} \, dy \quad\text{for } x \in \R^n \setminus \overline{\Omega},
	\end{equation*}
	which was introduced by Dipierro, Ros-Oton and Valdinoci in~\cite{drov}.  Problem~\eqref{eq:mainprob-frac} is known as the \textit{fractional torsion problem}. It is easy to see that it admits a unique weak solution~$u \in H^s(\R^n)$ -- see Section~\ref{sec:prelim} below --, and its analytic and probabilistic properties have received extensive attention in recent years.
	The local analogue of~\eqref{eq:mainprob-frac}, the classical torsion problem 
	\begin{equation}
		\label{eq:mainprob-serr}
		\begin{cases}
			\begin{aligned}
				& -\Delta u = 1		&& \text{in } \Omega, \\
				& u = 0 			&& \text{on } \partial \Omega,
			\end{aligned}
		\end{cases}
	\end{equation}
	has been analyzed by Serrin in his seminal paper~\cite{serrin-ARMA} together with an overdetermination on the normal derivative at the boundary, that is
	with the overdetermined condition
	\begin{equation}
		\label{eq:over-cond-serr}
		u_{\nu} = c_\flat < 0 \quad\text{on } \partial \Omega,
	\end{equation}
	where~$\nu$ denotes the exterior normal at the boundary. Serrin proved that if~$\Omega$ is a domain of class~$C^2$ such that there exists a classical solution~$u \in C^2(\Omega) \cap C^1(\overline{\Omega})$ to the overdetermined problem~\eqref{eq:mainprob-serr}--\eqref{eq:over-cond-serr}, then~$\Omega$ must be a ball. Later, Aftalion, Busca and Reichel~\cite{abr} performed a quantitative analysis of the moving plane method to establish a stability counterpart to Serrin's result, which was further refined in~\cite{cmv-holderstab}.
	
	On the other hand, Ciraolo, Magnanini and Sakaguchi~\cite{cms-parallel} investigated~\eqref{eq:mainprob-serr} in the setting of the interior parallel surface. Specifically, they imposed the condition that~$u$ is constant on a surface parallel to~$\partial\Omega$ and contained within~$\Omega$, also addressing the stability result for this problem. In particular, they considered domains of the form~$\Omega = G + B_R(0)$, with~$\partial G$ of class~$C^1$, and the overdetermined condition
	\begin{equation}
		\label{eq:over-cond-cms}
		u = c_\flat \quad\text{on } \partial G,
	\end{equation}
	proving that if there exists a classical solution~$u \in C^2(\Omega) \cap C^0(\overline{\Omega})$ to the problem~\eqref{eq:mainprob-serr} with the overdetermined condition~\eqref{eq:over-cond-cms}, then~$\Omega$ and~$G$ must be concentric balls. See also~\cite{sha-divserrin} for the symmetry result.
	
	The fractional torsion problem~\eqref{eq:mainprob-frac} has also been studied in the literature with various types of overdetermined conditions distinct from~\eqref{eq:over-cond}. First, Fall and Jarohs~\cite{fj} examined it together with a condition on the fractional normal derivative of~$u$ along~$\partial \Omega$ -- that is, a condition analogous to~\eqref{eq:over-cond-serr}. Specifically, they considered the fractional normal derivative 
	\begin{equation*}
		\left(\partial_\nu u\right)_s (x) \coloneqq - \lim_{t \to 0^+} \frac{u(x-t \nu(x))}{t^s} \qquad \text{for $x \in \partial \Omega$},
	\end{equation*}
	where~$\nu$ denotes the exterior normal, and they showed that, among all smooth bounded open sets~$\Omega$, balls are the only ones for which the solution~$u$ to~\eqref{eq:mainprob-frac} satisfies the overdetermined boundary condition
	\begin{equation*}
		\left(\partial_\nu u\right)_s = c_\flat < 0 \quad\text{on } \partial \Omega.
	\end{equation*}
	Subsequently,  Dipierro, Poggesi, Thompson and Valdinoci~\cite{dptv-serrin} established the quantitative counterpart of their result.
	
	Moreover, Ciraolo, Dipierro, Poggesi, Pollastro and Valdinoci~\cite{cir-pol-fractionaltorsion} investigated~\eqref{eq:mainprob-frac} in the setting of the interior parallel surface -- i.e., under the overdetermined condition~\eqref{eq:over-cond-cms}. Further developments concerning this type of overdetermination were later achieved by Dipierro, Poggesi, Thompson and Valdinoci~\cite{dptv-parallel,dptv-antisym}. \newline
	
	Despite the extensive literature on the topic, the problem under examination differs from previous works in two key aspects. First, the overdetermined condition~\eqref{eq:over-cond} does not appear to have been investigated before. Second, while we consider a parallel surface-type overdetermination, our prescribed surface lies in the exterior region of~$\Omega$. This distinction introduces significant difficulties compared to~\cite{cir-pol-fractionaltorsion,cms-parallel,dptv-parallel}, particularly when applying the moving plane method. 
	
	In contrast to the works~\cite{cir-pol-fractionaltorsion,cms-parallel,dptv-parallel}, our problem~\eqref{eq:mainprob-frac}--\eqref{eq:over-cond} is more closely related to an exterior problem. It is therefore worth noting that exterior overdetermined problems have already been studied in the local setting. For instance, Reichel~\cite{reichel-elec-cap,reichel-ext} adapted the method of moving planes to address Serrin-type problems on exterior domains. However, in those cases, the overdetermination is imposed on~$\partial\Omega$, which makes our problem fundamentally different from those in~\cite{reichel-elec-cap,reichel-ext}. A second crucial difference between our problem and an exterior problem is that, in our case, the equation for~$u$ holds only in~$\Omega$. As a result, the idea of reduced half-spaces developed in~\cite{reichel-ext} is not suitable for the present setting. This distinguishes the present paper also from the work of Soave and Valdinoci~\cite{soave-val}, who also employed the idea of reduced half-spaces in the study of some overdetermined exterior problems for the fractional Laplacian. \newline
	
	In the following two subsections, we present in detail the symmetry and stability results, along with our main assumptions.
	
	
	\subsection{Symmetry results.}
	
	As we shall see in the following, the moving plane method for the overdetermined problem~\eqref{eq:mainprob-frac}--\eqref{eq:over-cond} relies on specific assumptions on the geometry of~$\Omega$ and the distance parameter~$t>0$. Thus, before stating our first symmetry result, we recall the notion of sets of positive reach, introduced by Federer in the seminal paper~\cite{federer-curv}. This concept is essential for formulating and clarifying our assumptions.
	
	For a non-empty set~$E \subseteq \R^n$, we denote by~$\mathrm{Unp}(E)$ the set of all points~$x \in \R^n$ for which the projection
	\begin{equation*}
		\Pi_{E}(x) \coloneqq \left\{y \in E \mid \dist(x,E) = \abs*{x-y} \right\}
	\end{equation*}
	is a singleton. Moreover, for~$x \in E$, we define
	\begin{equation*}
		\mathrm{reach}(E,x) \coloneqq \sup \left\{r>0 \mid B_r(x) \subseteq \mathrm{Unp}(E)\right\}
	\end{equation*} 
	and the \textit{reach} of~$E$ as
	\begin{equation*}
		\mathrm{reach}(E) \coloneqq \inf_{x \in E} \mathrm{reach}(E,x).
	\end{equation*}
	If~$\mathrm{reach}(E)>0$, we say that~$E$ is a \textit{set of positive reach}, which necessarily implies that~$E$ is closed.
	
	In the remainder of this paper, we shall consider~\eqref{eq:mainprob-frac},~\eqref{eq:mainprob-frac-gen}, and~\eqref{eq:mainprob-frac-gen-relaxed} on a \textit{regular open set}~$\Omega \subseteq \R^n$ whose closure~$\overline{\Omega}$ has positive reach. Here, as usual, we call a set~$\Omega \subseteq \R^n$ \textit{regular open} if~$\Omega$ equals the interior of its closure. Moreover, a connected regular open set~$\Omega \subseteq \R^n$ will be called a regular domain in the following.  
	
	We can now state our first result.
	
	\begin{theorem}
		\label{th:symm-fractional}
		Let~$\Omega \subseteq \R^n$ be a regular open, bounded set whose closure~$\overline{\Omega}$ has positive reach~$r_{\Omega} > 0$, and let~$t \in (0,r_\Omega)$. If the unique solution $u\in H^{s}(\R^{n})$ of the fractional torsion problem~\eqref{eq:mainprob-frac} satisfies the overdetermined condition
		\begin{equation*}
			\mathcal{N}_s u = c_\flat \quad\text{on } \Gamma_t
		\end{equation*}
		with some constant $c_\flat \in \R$, then~$\Omega$ is a ball.
	\end{theorem}
	
	It is well known -- see, for instance,~\cite{dyda} -- that, for~$\Omega = B_r(x_0)$ where~$r>0$ and~$x_0 \in \R^n$, the solution to~\eqref{eq:mainprob-frac} can be computed explicitly and is given by
	\begin{equation*}
		\psi(x) \coloneqq \gamma_{n,s} \left(r^2-\abs*{x-x_0}^2\right)^{\! s}_{\!+},
	\end{equation*}
	with
	\begin{equation}
		\label{eq:gamma-ns}
		\gamma_{n,s} \coloneqq \frac{4^{-s} \,\Gamma(n/2)}{\Gamma(n/2+s) \,\Gamma(1+s)}.
	\end{equation}
	Since~$\psi$ is radial, it immediately follows that~$\mathcal{N}_s \psi$ is constant along every parallel surface to~$\partial\Omega$. Therefore, the converse of Theorem~\ref{th:symm-fractional} also holds.
	
	The proof of Theorem~\ref{th:symm-fractional} proceeds via the moving plane method and shares some features with that of the nonlocal Alexandrov's theorem established by Cabré, Fall, Solà-Morales and Weth~\cite{cabre-weth}, as well as by Ciraolo, Figalli, Maggi and Novaga~\cite{cfmnov} at the same time.
	
	In order to carry out the moving plane method, it is essential to note that, for a fixed~$t>0$, the outer parallel set~$\Gamma_t$ can be written as~$\Gamma_t= \partial G_t$ with
	\begin{equation}
		\label{eq:defG}
		G_t \coloneqq \left\{x \in \R^n \mid \mathrm{dist}(x,\Omega)<t\right\} \!.
	\end{equation}
	Moreover, the assumption~$t<r_{\Omega}$ guarantees, by Theorem~4.8 in~\cite{federer-curv}, that~$\partial G_t$ is at least of class~$C^{1,1}$. Thus, the set~$G_t$ has sufficient regularity to ensure the geometric features of the critical position in the moving plane procedure. See, for instance, Section~5.2 in~\cite{frae-book} for details.
	
	We also note that, by Remark~3.2 in~\cite{rataj-posreach}, the projection~$\Pi_{\overline{\Omega}}$ defines a retraction of~$G_t$ onto~$\overline{\Omega}$. It then follows that the fundamental and homology groups of~$G_t$ are finitely generated. As a consequence,~$\overline{\Omega}$ and~$\R^n \setminus \overline{\Omega}$ only have finitely many connected components.
	
	Finally, we note that the assumption~$t<r_{\Omega}$ is fundamental for using an equation on~$\Omega$ within the moving plane machinery, as will become clear in Lemma~\ref{lem:inclusion} below and the explanation preceding it. For an example where these assumptions are violated, resulting in an undesired configuration, see Figure~\ref{fig:count-posreach} below.
	
	Remarkably, we observe that no restriction on $t>0$ is necessary within the class of convex domains. Indeed, if~$\Omega$ is a convex domain, then $\Omega$ is regular and~$r_{\Omega} = +\infty$. Thus, any outer parallel surface to~$\partial \Omega$ is admissible in this case.
	
	We also extend our analysis to a more general class of problems, specifically considering the equation
	\begin{equation}
		\label{eq:mainprob-frac-gen}
		\begin{cases}
			\begin{aligned}
				& \!\left(-\Delta\right)^s u = f(u)	&& \text{in } \Omega, \\
				& u = 0 							&& \text{in } \R^n \setminus \Omega, \\
				& u > 0 							&& \text{in } \Omega,
			\end{aligned}
		\end{cases}
	\end{equation}
	where
	\begin{equation}
		\label{eq:ass-f-frac}
		f:[0,+\infty) \to \R \quad\text{with } f \in C^{0,1}_{\loc}([0,+\infty)).
	\end{equation}
	
	In this case, we will prove the following theorem.
	
	\begin{theorem}
		\label{th:sym-frac-gen}
		Let~$\Omega \subseteq \R^n$ be a regular open, bounded set whose closure~$\overline{\Omega}$ has positive reach~$r_{\Omega} > 0$, and let $t \in (0,r_{\Omega})$. If $f$ satisfies~\eqref{eq:ass-f-frac} and there exists a solution~$u \in H^s(\R^n) \cap L^\infty(\R^n)$ to~\eqref{eq:mainprob-frac-gen} satisfying the overdetermined condition
		\begin{equation*}
			\mathcal{N}_s u = c_\flat \quad\text{on }\Gamma_t
		\end{equation*}
		with a constant $c_\flat \in \R$, then~$\Omega$ is a ball.
	\end{theorem}

	To maintain clarity, we present the proof of Theorem~\ref{th:symm-fractional} in full detail before addressing Theorem~\ref{th:sym-frac-gen} in Section~\ref{sec:gen-sym-frac}. Indeed, the proof of Theorem~\ref{th:symm-fractional} already brings out the key structural similarities and differences with respect to~\cite{cabre-weth,cfmnov,fj}, while the argument for Theorem~\ref{th:sym-frac-gen} consists mainly of a technical adaptation of the previous one.
	
	
	\subsection{Stability results.}
	
	Once the symmetry result has been established, it is natural to address its quantitative counterpart. The main question is whether~$\mathcal{N}_s u$ being almost constant along~$\Gamma_t$ implies that~$\Omega$ is close to being a ball. As in~\cite{abr}, the deviation of~$\Omega$ from being a ball is measured using the deficit
	\begin{equation*}
		\rho(\Omega) \coloneqq \inf\left\{R-r \mid B_r(x) \subseteq \Omega \subseteq B_R(x) \quad\text{for some } x \in \Omega\right\}\!,
	\end{equation*}
	while the discrepancy of~$\mathcal{N}_s u$ from being constant along~$\Gamma_t$ is quantified by the semi-norm
	\begin{equation*}
		\left[\mathcal{N}_s u\right]_{\sGt} \coloneqq \sup_{\substack{x,y \in \Gamma_t \\ x \neq y}} \,\frac{\abs*{\mathcal{N}_s u(x)-\mathcal{N}_s u(y)}}{\abs*{x-y}}.
	\end{equation*}
	
	Our first stability result provides a quantitative bound for~$\rho(\Omega)$ in terms of~$\left[\mathcal{N}_s u\right]_{\sGt}$ and is as follows.
	
	\begin{theorem}
		\label{th:stab-fractional}
		Assume that~$n \geq 2$ or~$n=1$ and~$s \geq 1/2$. Moreover, suppose that~$\Omega \subseteq \R^n$ is a regular open, bounded set satisfying the uniform interior sphere condition with radius~$\mathfrak{r}_\Omega>0$, and that~$\overline{\Omega}$ has positive reach~$r_{\Omega} > 0$. Then, the unique solution~$u \in H^s(\R^n)$ of the fractional torsion problem~\eqref{eq:mainprob-frac} satisfies
		\begin{equation*}
			\label{eq:stab-est-fractional}
			\rho(\Omega) \leq C \left[\mathcal{N}_s u\right]_{\sGt}^{\frac{1}{2+s}} \!,
		\end{equation*}
		for every~$t \in (0,r_{\Omega})$, where 
		\begin{gather*}
			C \coloneqq 16 \left(n+3\right) \left(s+2\right) \frac{\mathrm{diam}(\Omega)}{\mathfrak{r}_{\Omega}^n \,\abs*{B_1}} \, C_1^{\frac{1}{2+s}} \left(\frac{C_2}{s+1}\right)^{\!\frac{1+s}{2+s}}\!, \\
			\notag
			C_1 \coloneqq \frac{\left(\mathrm{diam}(\Omega) + r_{\Omega}\right)^{n+2s+2}}{c_{n,s} \,\gamma_{n,s} \, \mathfrak{r}_\Omega^s \left(n+2s\right)}, \quad C_2 \coloneqq \mathrm{diam}(\Omega)^{n-1}+\frac{n \,\abs*{B_1} \,\mathrm{diam}(\Omega)^{n}}{2^{n-1} \,\mathfrak{r}_{\Omega}}.
		\end{gather*}
	\end{theorem}
	
	It is worth noting that, since we assume that~$\Omega$ is a set of positive reach satisfying the uniform interior sphere condition, it must also satisfy a uniform two-sided sphere condition and hence be of class~$C^{1,1}$. See, for instance, Section~2 in~\cite{dalph-ball} for more details. \newline
	
	Also in this case, we can extend the previous stability result to a more general class of problems. Specifically, we have the subsequent theorem.
	
	\begin{theorem}
		\label{th:stab-fractional-gen}
		Assume that~$n \geq 2$ or~$n=1$ and~$s \geq 1/2$. Suppose also that~$\Omega \subseteq \R^n$ is a regular open, bounded set satisfying the uniform interior sphere condition with radius~$\mathfrak{r}_\Omega>0$, and that~$\overline{\Omega}$ has positive reach~$r_{\Omega} > 0$. Moreover, let~$f$ satisfy~\eqref{eq:ass-f-frac} and~$f(0) \geq 0$. Then, for every non-trivial solution~$u \in H^s(\R^n) \cap L^\infty(\R^n)$ to
		\begin{equation}
			\label{eq:mainprob-frac-gen-relaxed}
			\begin{cases}
				\begin{aligned}
					& \!\left(-\Delta\right)^s u = f(u)	&& \text{in } \Omega, \\
					& u = 0 							&& \text{in } \R^n \setminus \Omega, \\
					& u \geq 0 							&& \text{in } \Omega,
				\end{aligned}
			\end{cases}
		\end{equation}
		and every~$t \in (0,r_{\Omega})$, we have 
		\begin{equation*}
			\rho(\Omega) \leq C \left[\mathcal{N}_s u\right]_{\sGt}^{\frac{1}{2+s}} \!,
		\end{equation*}
		where~$C, C_2>0$ are as in Theorem~\ref{th:stab-fractional}, and
		\begin{equation*}
			C_1 \coloneqq \frac{\mathrm{diam}(\Omega)^{n+2s+2}}{c_{n,s} \, C' \, \mathfrak{r}_\Omega^s \left(n+2s\right)},
		\end{equation*}
		with~$C'>0$ given by~\eqref{eq:def-Cast-bounbelowu} below.
	\end{theorem}
	
	To maintain parallelism, the proof of Theorem~\ref{th:stab-fractional-gen} is also postponed to Section~\ref{sec:gen-sym-frac}, as it is a variation of the previous argument, making use of tools developed in~\cite{dptv-parallel}. \newline
	
	It should not be surprising that the uniform interior sphere condition is assumed in Theorems~\ref{th:stab-fractional} and~\ref{th:stab-fractional-gen}, since, as far as we know, all available stability results in the context of fractional overdetermined problems rely on regularity assumptions of this type. However, in view of our symmetry results, it is reasonable to expect, in this context, stability also for a larger class of domains, including general convex domains and Lipschitz domain with corners, such as that in Figure~\ref{subfig:dom-casei} below. Since the interior sphere condition enters in an essential and quantitative way in Theorems~\ref{th:stab-fractional} and~\ref{th:stab-fractional-gen}, we need a new approach to deal with such more general domains. In Section~\ref{sec:other-stab-result}, we present a possible approach to remove the interior sphere condition at the cost of a reduced stability exponent. Specifically, we will prove the following results for general bounded regular domains whose closures have positive reach.
	
	\begin{theorem}
		\label{th:no-int-sphere}
		Assume that~$n \geq 2$, let~$\Omega \subseteq \R^n$ be a regular bounded domain, and suppose that the closure~$\overline{\Omega}$ has positive reach~$r_{\Omega} > 0$. Then, the unique solution~$u \in H^s(\R^n)$ of the fractional torsion problem~\eqref{eq:mainprob-frac} satisfies 
		\begin{equation*}
			\rho(\Omega) \leq C_{\Omega} \left[\mathcal{N}_s u\right]_{\sGt}^{\frac{1}{2+2s}} \!,
		\end{equation*}
		for every~$t \in (0,r_{\Omega})$, where
		\begin{gather*}
			C_{\Omega} \coloneqq \max\left\{32 \,\frac{\left(n+3\right) \left(s+1\right)}{\abs*{\Omega}} \, C_1^{\frac{1}{2+2s}} \left(\frac{C_2}{2s+1}\right)^{\!\frac{2s+1}{2+2s}} \!, \left(\frac{\left(2s+1\right) C_1}{\mathrm{diam}(\Omega)^{n-1}}\right)^{\!\frac{1}{2+2s}} \right\} \mathrm{diam}(\Omega), \\
			\notag
			C_1 \coloneqq \frac{\left(\mathrm{diam}(\Omega) + r_{\Omega}\right)^{n+2s+2}}{c_{n,s} \,\gamma_{n,s} \left(n+2s\right)}, \quad C_2 \coloneqq \mathrm{diam}(\Omega)^{n-1} + \left[1+\frac{2}{r_\Omega}\right]^{n-1} \Phi_{n-1}(\overline{\Omega},\R^n).
		\end{gather*}
		Here,~$\Phi_{n-1}(\overline{\Omega},\cdot)$ denotes the~$(n-1)$-th curvature measure associated with~$\overline{\Omega}$ -- see~\cite{federer-curv}.
	\end{theorem}
	
	\begin{theorem}
		\label{th:stab-gen-nosphere}
		Assume that~$n \geq 2$, and let~$\Omega \subseteq \R^n$ be a regular domain whose closure~$\overline{\Omega}$ has positive reach~$r_{\Omega} > 0$. Moreover, let~$f$ satisfy~\eqref{eq:ass-f-frac} and~$f(0) \ge 0$. Then, for every non-trivial solution~$u \in H^s(\R^n) \cap L^\infty(\R^n)$ of 
		\begin{equation*}
			\begin{cases}
				\begin{aligned}
					& \!\left(-\Delta\right)^s u = f(u)	&& \text{in } \Omega, \\
					& u = 0 							&& \text{in } \R^n \setminus \Omega, \\
					& u \geq 0 							&& \text{in } \Omega,
				\end{aligned}
			\end{cases}
		\end{equation*}
		we have 
		\begin{equation*}
			\rho(\Omega) \leq C_\Omega \left[\mathcal{N}_s u\right]_{\sGt}^{\frac{1}{2+2s}} \!,
		\end{equation*}
		for every $t \in (0,r_{\Omega})$, where~$C_\Omega, C_2>0$ are as in Theorem~\ref{th:no-int-sphere}, and
		\begin{equation*}
			C_1 \coloneqq \frac{\left(\mathrm{diam}(\Omega) + r_{\Omega}\right)^{n+2s+2}}{c_{n,s} \, C' \left(n+2s\right)},
		\end{equation*}
		with~$C'>0$ given by~\eqref{eq:def-Cast-bounbelowu} below.
	\end{theorem}
	
	We conclude by noting that, beside the fact that~$\overline{\Omega}$ has positive reach, no regularity assumptions are imposed on~$\partial \Omega$ in Theorems~\ref{th:symm-fractional},~\ref{th:sym-frac-gen},~\ref{th:no-int-sphere}, and~\ref{th:stab-gen-nosphere}. To the best of our knowledge, in the context of local Serrin-type problems, the most refined result in terms of domain regularity has recently been obtained by Figalli and Zhang in~\cite{fz-serrin}. More precisely, they adapted Weinberger’s proof~\cite{weinb-ARMA} to bounded, indecomposable sets of finite perimeter satisfying a density condition -- a class that includes Lipschitz domains -- by exploiting techniques from Geometric Measure Theory.
	
	Very recently, Domingo-Pasarin and Ros-Oton~\cite{dprs-altcaf} provided an alternative proof of Serrin’s result for Lipschitz domains by establishing the smoothness of Lipschitz free boundaries for weak solutions to a one-phase Bernoulli problem. Exploiting this fact, they showed that a Lipschitz domain supporting a weak solution to Serrin's problem~\eqref{eq:mainprob-serr}--\eqref{eq:over-cond-serr} is actually smooth, together with the corresponding solution, so that Serrin's original result~\cite{serrin-ARMA} applies.
	
	Therefore, the absence of specific regularity assumptions in our setting appears to be a genuinely nonlocal phenomenon.
	
	
	\subsection{Structure of the paper.}
	
	In Section~\ref{sec:prelim}, we collect some preliminary results on the fractional setting and closed sets of positive reach which will be used later. In Section~\ref{sec:proof-sym-fractional}, we then introduce a key notation related to a variant of the moving plane method which we use, and we complete the proof of Theorem~\ref{th:symm-fractional}. Section~\ref{sec:proof-stabil-fractional} is devoted to the proof of the stability result given in Theorem~\ref{th:stab-fractional}. Here, we use the notation and some inequalities from the previous section, and we perform a refined analysis to get the necessary quantitative information. In Section~\ref{sec:gen-sym-frac}, we then prove some extensions of Theorems~\ref{th:symm-fractional} and~\ref{th:stab-fractional} to the semilinear setting, as stated in Theorems~\ref{th:sym-frac-gen} and~\ref{th:stab-fractional-gen}. Finally, in Section~\ref{sec:other-stab-result}, we prove the stability results, given in Theorems~\ref{th:no-int-sphere} and~\ref{th:stab-gen-nosphere}, under significantly weaker regularity assumptions on~$\Omega$.
	
	
	\section{Preliminaries}
	\label{sec:prelim}
	
	This section is divided into three brief subsections, where we collect some notions and results that will be used later.
	
	
	\subsection{Fractional setting.}
	
	Let~$n \geq 1$ and~$s \in (0,1)$. The fractional Sobolev space~$H^s(\R^n)$ is defined as
	\begin{equation*}
		H^s(\R^n) \coloneqq \left\{u \in L^2(\R^n) \,\middle|\, [u]_{H^s(\R^n)} < +\infty \right\}\!,
	\end{equation*}
	where~$[\cdot]_{H^s(\R^n)}$ is the Gagliardo semi-norm given by
	\begin{equation*}
		[u]_{H^s(\R^n)}^2 \coloneqq c_{n,s} \int_{\R^n} \int_{\R^n} \frac{\abs*{u(x)-u(y)}^2}{\abs*{x-y}^{n+2s}} \, dx \, dy.
	\end{equation*}
	Furthermore, for~$u,v \in H^s(\R^n)$, the bilinear form associated with the fractional Laplacian is given by
	\begin{equation*}
		\mathcal{E}(u,v) \coloneqq \frac{c_{n,s}}{2} \int_{\R^n} \int_{\R^n} \frac{\left(u(x)-u(y)\right)\left(v(x)-v(y)\right)}{\abs*{x-y}^{n+2s}} \, dx \, dy.
	\end{equation*}
	
	For any arbitrary open, bounded set~$\Omega \subseteq \R^n$, we also define the space
	\begin{equation*}
		\cH^s_0(\Omega) \coloneqq \left\{u \in H^s(\R^n) \mid u=0 \text{ on } \R^n \setminus \Omega \right\}\!.
	\end{equation*}
	For~$g \in L^2(\Omega)$ and~$c \in L^\infty(\Omega)$, we say that a function~$u \in H^s(\R^n)$ satisfies 
	\begin{equation*}
		\label{eq:eq-def-sol}
		\left(-\Delta\right)^s u + c \, u =  g \quad\text{in } \Omega
	\end{equation*}
	in \textit{weak sense} if  
	\begin{equation*}
		\label{eq:def-sol-weak}
		\mathcal{E}(u,\phi) + \int_{\Omega} c \, u \phi \, dx =  \int_{\Omega} g \phi \, dx \quad\text{for all } \phi \in \cH^s_0(\Omega).
	\end{equation*}
	Moreover, we say that~$u$ satisfies the inequality 
	\begin{equation*}
		\label{eq:eq-def-sol-1}
		\left(-\Delta\right)^s u + c \, u \ge  g \quad\text{in } \Omega
	\end{equation*}
	in weak sense if 
	\begin{equation*}
		\label{eq:def-supersol-weak}
		\mathcal{E}(u,\phi) + \int_{\Omega} c \, u \phi \, dx \ge  \int_{\Omega} g \phi \, dx \quad\text{for all non-negative } \phi \in \cH^s_0(\Omega),
	\end{equation*}
	and we use the same notation for the reversed inequality sign.
	
	In particular, we see that weak solutions~$u \in H^s(\R^n)$ of~\eqref{eq:mainprob-frac} are precisely the critical points~$u \in \cH^s_0(\Omega)$ of the strictly convex and coercive functional
	\begin{equation*}
		\mathcal{J} \in C^1(\cH^s_0(\Omega),\R), \qquad \quad \mathcal{J}(u) \coloneqq \frac{1}{2} \,\mathcal{E}(u,u)- \int_{\Omega} u \, dx .
	\end{equation*}
	Note here that the coercivity of~$\mathcal{J}$ follows from the fractional Poincaré inequality since~$\Omega$ is bounded. Hence,~$\mathcal{J}$ has precisely one critical point which is the global minimum of~$\mathcal{J}$, and therefore~\eqref{eq:mainprob-frac} has a unique solution.
	
	
	\subsection{Remarks on solutions to the overdetermined problems.}
	
	The unique solution~$u \in H^s(\R^n)$ to~\eqref{eq:mainprob-frac} must be of class~$C^{\infty}(\Omega) \cap L^\infty(\R^n)$, as follows from Propositions~2.2 and~2.3 of~\cite{ro-serra}, the bootstrap argument, and Claim~2.8 in~\cite{ro-serra}. Furthermore, if we assume that~$\Omega$ is Lipschitz and~$r_\Omega>0$, then we also get that~$u \in C^{s}(\R^n)$ by Proposition~1.1 in~\cite{ro-serra}. Moreover, we have 
	\begin{equation}
		\label{eq:strict-posi}
		u > 0 \quad \text{in } \Omega.
	\end{equation}
	In fact, this follows without any regularity assumptions on~$\Omega$. More precisely, by Lemma~5.3 in~\cite{cir-pol-fractionaltorsion} -- which is an application of the weak comparison principle in~\cite{fj} -- we obtain
	\begin{equation}
		\label{eq:u-bbelow-weak}
		u(x) \geq \gamma_{n,s} \,\mathrm{dist}(x,\partial\Omega)^{2s} \quad\text{for every } x \in \Omega.
	\end{equation}
	
	Regarding the regularity of solutions~$u \in H^s(\R^n) \cap L^\infty(\R^n)$ to~\eqref{eq:mainprob-frac-gen} under the assumption~\eqref{eq:ass-f-frac}, Propositions~2.2 and~2.3 in~\cite{ro-serra}, together with a possible bootstrap argument, imply that~$u \in C^{2s+1-\varepsilon}(\Omega)$ with~$\varepsilon \in (0,2s)$ and~$2s - \varepsilon \not\in \Z$. See, for instance, Remark~4.1 in~\cite{dptv-parallel} for details.
	
	Furthermore, for non-trivial bounded solutions to~\eqref{eq:mainprob-frac-gen} under the hypothesis~\eqref{eq:ass-f-frac}, if we additionally assume that~$f(0) \geq 0$, then Proposition~3.1 in~\cite{dptv-parallel} ensures that
	\begin{equation}
		\label{eq:u-bbelow-weak-2}
		u(x) \geq C' \,\mathrm{dist}(x,\partial\Omega)^{2s} \quad\text{for every } x \in \Omega,
	\end{equation}
	where
	\begin{multline}
		\label{eq:def-Cast-bounbelowu}
		C' \coloneqq \frac{C_{n,s}}{\max\left\{1,\mathrm{diam}(\Omega)\right\}^{n+2s}} \left(1+\mathrm{diam}(\Omega)^{2s} \,\norma*{f}_{C^{0,1}\left([0,\norma*{u}_{L^\infty(\R^n)}]\right)}\right)^{\! -1} \cdot
		\\ \cdot \left(f(0) + \int_{\R^n} \frac{u(x)}{1+\abs*{x}^{n+2s}} \, dx\right)
	\end{multline}
	for some~$C_{n,s}>0$. Observe that~$C'>0$ since~$u$ is non-trivial. In particular, this implies that the positivity condition in~\eqref{eq:mainprob-frac-gen} can be relaxed to~$u \geq 0$.
	
	Finally, we observe that, from~\eqref{eq:strict-posi} or the positivity assumption in~\eqref{eq:mainprob-frac-gen}, and taking advantage of the Dirichlet conditions in~\eqref{eq:mainprob-frac} and~\eqref{eq:mainprob-frac-gen}, it follows that
	\begin{equation}
		\label{eq:expr-nonlocal-normal}
		\mathcal{N}_s u(x) = -c_{n,s} \int_{\Omega} \frac{u(y)}{\abs*{x-y}^{n+2s}} \, dy<0 \quad\text{for all } x \in \R^n \setminus \overline{\Omega},
	\end{equation}
	which, in turn, implies~$c_\flat<0$ in~\eqref{eq:over-cond}.
	
	
	\subsection{The tangent and normal cones.}
	
	We recall some geometric notions, introduced by Federer in~\cite{federer-curv}, which are needed in the following.
	
	For~$E \subseteq \R^n$ and~$x \in E$, we denote by~$\mathrm{Tan}(E,x)$ the set of all tangent vectors of~$E$ at~$x$, consisting of all~$u \in \R^n$ such that
	\begin{gather*}
		\text{either } u=0 \text{ or there exist sequences } \{x_k\}_k \subseteq E \setminus \{x\} \text{ and } \{r_k\}_k \subseteq (0,+\infty) \\
		\text{such that } x_k \to x \text{ and } r_k\left(x_k-x\right) \to u \text{ as } k \to +\infty.
	\end{gather*}
	The set~$\mathrm{Tan}(E,x)$ is a closed cone called the \textit{tangent cone}. Its dual cone, defined by
	\begin{equation}
		\label{eq:def-Nor-cone}
		\mathrm{Nor}(E,x) \coloneqq \left\{v \in \R^n \mid u \cdot v \leq 0 \text{ for all } u \in \mathrm{Tan}(E,x) \right\} \!,
	\end{equation}
	is called the \textit{normal cone}.
	
	The following lemma collects two useful results, which can be found in Theorem~4.8 of~\cite{federer-curv}, in Proposition~3.1 of~\cite{rataj-posreach}, or in Corollaries~4.6 and~4.12 of~\cite{rataj-book}.
	
	\begin{lemma}
		\label{lem:property-posreach}
		Assume that~$E \subseteq \R^n$,~$x \in E$, and~$\mathrm{reach}(E,x)>0$.
		\begin{enumerate}[leftmargin=*,label=$(\arabic*)$]
			\item \label{it:Federer-ineq} If~$v \in \mathrm{Nor}(E,x)$ and~$y \in E$, then
			\begin{equation*}
				v \cdot (y-x) \leq \frac{\abs*{y-x}^2 \,\abs*{v}}{2 \,\mathrm{reach}(E,x)}.
			\end{equation*}
			\item \label{it:Nor-nonemp} $\mathrm{Nor}(E,x)$ is non-trivial, meaning that~$\{0\} \subsetneq \mathrm{Nor}(E,x) \subsetneq \R^n$, if and only if~$x \in \partial E$.
		\end{enumerate}
	\end{lemma}
	
	
	\section{Proof of the symmetry result for the torsion problem}
	\label{sec:proof-sym-fractional}
	
	In this section we will provide the proof of Theorem~\ref{th:symm-fractional}. 
	Let~$\Omega \subseteq \R^n$,~$u \in H^s(\R^n)$, and~$t \in (0,r_\Omega)$ satisfy the assumptions of this theorem. With the help of the moving plane method, we shall prove in the following that there exists~$\lambda_\star \in \R$ with the property that~$\Omega$ is symmetric with respect to the reflection across the hyperplane~$\{x \in \R^n \mid x_1 = \lambda_\star\}$. Of course, by the rotational invariance of the problem, we can replace the~$e_1$-direction by any other direction in this argument, and a standard procedure then yields that~$\Omega$ must be ball.
	
	In order to perform the moving plane method in the~$e_1$-direction, we need some notation. First, since~$t \in (0,r_\Omega)$ is fixed, we shall simply write~$G$ for the set~$G_t$ defined by~\eqref{eq:defG} in the following. For a given~$\lambda \in \R$, we also define the sets
	\begin{equation*}
		H_\lambda \coloneqq \left\{x \in \R^n \mid x_1>\lambda\right\}, \quad T_\lambda \coloneqq  \left\{x \in \R^n \mid x_1=\lambda\right\}, \quad \Omega_\lambda \coloneqq \Omega \cap H_\lambda,\quad G_\lambda \coloneqq G \cap H_\lambda.
	\end{equation*}
	Moreover, we indicate the reflection of a point~$x \in \R^n$ across the hyperplane~$T_\lambda$ by
	\begin{equation*}
		\sigma_{\lambda}(x) \coloneqq x + 2\left(\lambda - e_1 \cdot x\right)e_1 = (2\lambda-x_1,x_2,\dots,x_n).
	\end{equation*}
	We also denote by
	\begin{equation*}
		\Omega^\lambda \coloneqq \sigma_\lambda(\Omega),\quad G^\lambda \coloneqq \sigma_\lambda(G)\quad \text{and}\quad H^\lambda \coloneqq \sigma_\lambda(H_\lambda)= \left\{x \in \R^n \mid x_1<\lambda\right\}
	\end{equation*}
	the reflections of~$\Omega$ and~$G$ across the hyperplane~$T_\lambda$, respectively, and we set
	\begin{equation*}
		\Omega^\lambda_\sH \coloneqq \Omega^\lambda \cap H^\lambda, \quad G^\lambda_\sH \coloneqq G^\lambda \cap H^\lambda.
	\end{equation*}
	Finally, we write
	\begin{equation*}
		u_\lambda(x) \coloneqq u(\sigma_{\lambda}(x)) \quad \text{for } x \in \R^n.
	\end{equation*}
	Setting~$\Lambda \coloneqq \sup \left\{x_1 \mid x \in G\right\}$, we then define the \textit{critical value} associated with the~$e_1$-direction as
	\begin{equation}
	\label{eq:lambda-crit}
		{\lambda}_\star \coloneqq \inf \left\{\lambda \in \R \mid G^\mu_\sH \subseteq G \text{ for all } \mu \in (\lambda,\Lambda)\right\} \!.
	\end{equation}
	For simplicity, we shall replace~$\lambda_\star$ by~$\star$ in the notation introduced above, so we write $H_\star$,~$T_\star$,~$\Omega_\star$,~$G_\star$,~$\sigma_\star$,~$\Omega^\star$,~$G^\star$,~$H^\star$,~$\Omega^\star_\sH$,~$G^\star_\sH$, and~$u_\star$ in place of~$H_{\lambda_\star}$,~$T_{\lambda_\star}$,~$\Omega_{\lambda_\star}$~$G_{\lambda_\star}$,~$\sigma_{\lambda_\star}$,~$\Omega^{\lambda_\star}$, $G^{\lambda_\star}$,~$H^{\lambda_\star}$,~$\Omega^{\lambda_\star}_\sH$,~$G^{\lambda_\star}_\sH$, and~$u_{\lambda_\star}$. We refer to this situation as the \textit{critical position}, and we call~$T_\star$ the \textit{critical hyperplane}.

	From the definition of~$\lambda_\star$, we deduce the inclusion~$G^{\star}_\sH \subseteq \overline{G}$, which then yields 
	\begin{equation}
		\label{eq:extra-inclusion}
		G^{\star}_\sH \subseteq G                
	\end{equation}
	since~$G$ is of class~$C^{1,1}$ and therefore regular open. Moreover, it is well known -- see, for instance,~\cite[Section~5.2]{frae-book} for a proof -- that 
	the critical position~$\lambda_\star$ gives rise to one of the following key geometric implications: 
	\begin{enumerate}[leftmargin=*, label=\textbf{Case \arabic*}]
		\item \label{it:int-touch} there exists a point~$x_\star \in \partial G \cap \partial G^\star \cap H^{\star}$, called an \textit{interior touching point} for the inclusion~\eqref{eq:extra-inclusion};
		\item \label{it:non-transv-inter} there exists a point~$x_\star \in T_{\star} \cap \partial G \cap \partial G^\star$ with~$e_1 \in T_{x_\star} \partial G \cap T_{x_\star} \partial G^\star$, called \textit{non-transversal intersection point}. Here,~$T_{x_\star} \partial G$ and~$T_{x_\star} \partial G^\star$ denote the tangent spaces of~$\partial G$ and~$\partial G^\star$ at~$x_\star$, respectively.  \end{enumerate}
	
	While it is easy to see that the definition of~$\lambda_\star$ implies that~$G_{\star} \not = \varnothing$, this property fails in general for $\Omega$. We will therefore distinguish the cases
	\begin{equation*}
	\label{eq:case-dist-omega}
		\Omega_{\star} = \varnothing \qquad \text{and}\qquad \Omega_{\star} \not = \varnothing   
	\end{equation*}
	at some places in the following arguments.
	We refer to Figure~\ref{subfig:dom-casei} for an example of a regular open set~$\Omega$ whose closure has positive reach and for which~$\Omega_{\star} = \varnothing$.
    \begin{figure}
			\centering
			\begin{tikzpicture}[x=1pt,y=1pt]
				\draw (-50,-50) -- (50,50);
				\draw (50,50) -- (50,-50);
				\draw (50,-50) -- (-50,-50);
				
				\draw[densely dotted] (75,50) arc (0:90:25);
				\draw[densely dotted] (50,-75) arc (270:360:25);
				
				\draw[densely dotted] (50,-75) -- (-50,-75);
				\draw[densely dotted] (75,50) -- (75,-50);
				\draw[densely dotted] (32.3223304703363,67.6776695296637) -- (-67.6776695296637,-32.3223304703363);
				
				\draw[densely dotted] (-67.6776695296637,-32.3223304703363) arc (135:270:25);
				
				\draw[dashed] (50,100) -- (50,-100);
				\draw[densely dotted, name path=A] (50,75) arc (90:180:25);
				\draw[densely dotted, name path=B] (25,-50) arc (180:270:25);
				\draw[densely dotted] (25,50) -- (25,-50);
				\draw[->] (50,75)--(30,75) node[left]{$e_1$};
				
				\draw (51,-98) node [anchor=west][inner sep=0.75pt]  [align=center] {$T_{\star}$};
				\draw (-41,-50) node [anchor=south west][inner sep=0.75pt]  [align=center] {$\Omega$};
				\draw (-68,-68) node [anchor=south west][inner sep=0.75pt]  [align=center] {$G$};
				\draw (54,0) node [anchor=west][inner sep=0.75pt]  [align=center] {$G_{\star}$};
				\draw (29,0) node [anchor=west][inner sep=0.75pt]  [align=center] {$G^{\star}_\sH$};
				
				\begin{scope}
					\clip (50,75) -- (25,75) -- (25,-75) -- (50,-75)-- cycle;
					\tikzfillbetween[of=A and B]{black, opacity=0.2};
				\end{scope}
			\end{tikzpicture}
			\caption{The plane stops upon touching $\Omega$, so $\Omega_\star = \varnothing$. At the critical position, both \ref{it:int-touch} and \ref{it:non-transv-inter} occur simultaneously.}
			\label{subfig:dom-casei}
     \end{figure}
	
	In addition to the inclusion~\eqref{eq:extra-inclusion} which readily follows from our definition of~$\lambda_\star$, we also need the inclusion $\Omega^{\star}_\sH \subseteq \Omega$, since we do not have an equation for the difference function~$u-u_{\star}$ in~$\Omega^{\star}_\sH \setminus \Omega$. Figure~\ref{fig:count-posreach-B} illustrates that this inclusion property does not hold in general. However, as noted in the subsequent key lemma, our assumption~$t \in (0,r_\Omega)$ implies the required inclusion. This result is purely geometric and not related to the equation under consideration, moreover we believe it could be of independent interest.
	
	\begin{figure}
		\centering
		\begin{subfigure}{0.48\textwidth}
			\centering
			\begin{tikzpicture}[x=1.3pt,y=1.3pt]
				
				\fill[black!20] 
				(17.6776695296637, -17.6776695296637) arc (-45:45:25) -- 
				(17.6776695296637, 17.6776695296637) arc (225:45:10) -- 
				(31.8198051533946, 31.8198051533946) arc (45:-45:45) -- 
				(31.8198051533946, -31.8198051533946) arc (315:135:10) -- cycle;
				
				\draw (17.6776695296637, -17.6776695296637) arc (-45:45:25);
				\draw (17.6776695296637, 17.6776695296637) arc (225:45:10);
				\draw (31.8198051533946, 31.8198051533946) arc (45:-45:45);
				\draw (31.8198051533946, -31.8198051533946) arc (315:135:10);
				
				\draw (53.0330085889911, 53.0330085889911) arc (45:-45:75);
				\draw (53.0330085889911, 53.0330085889911) arc (45:218.2226079814105:40);
				\draw (53.0330085889911, -53.0330085889911) arc (315:141.7773920185895:40);
				
				
				\begin{scope}
					\clip (10,50) -- (25,50) -- (25,-50) -- (10,-50)-- cycle;
					\draw[densely dotted] (-6.6757753834122, 0) circle (30);
				\end{scope}
				
				\draw[densely dotted] (-6.6757753834122, 0) -- (16.8926091602938, -18.5615530061469);
				\draw[densely dotted] (-6.6757753834122, 0) -- (16.8926091602938, 18.5615530061469);
				
				\draw (30,-5) node [anchor=south west][inner sep=0.75pt]  [align=center] {$\Omega$};
				\draw (55,-5) node [anchor=south west][inner sep=0.75pt]  [align=center] {$G$};
				\draw (-7.7, 0) node [anchor=east][inner sep=0.75pt]  [align=center] {$x_0$};
			\end{tikzpicture}
			\caption{An example where $\partial G$ is too far from $\overline{\Omega}$. The point $x_0 \in \partial G$ has two projections on $\overline{\Omega}$, and $\partial G$ is not even $C^1$.}
			\label{fig:count-posreach}
		\end{subfigure}
		\hfill
		\begin{subfigure}[b]{0.48\textwidth}
			\centering
			\begin{tikzpicture}[x=1.3pt,y=1.3pt]
				
				\fill[black!20] 
				(34.1621123082939, 63.6253160456164) arc (103.611333799743:135:40) -- 
				(15.2912160275968, 53.0330085889911) arc (135:225:75) -- 
				(15.2912160275968, -53.0330085889911) arc (225:256.388666200257:40) -- 
				(34.1621123082939, 29.2907849439976) arc (139.3900296688734:220.6099703311267:45) -- cycle;
				
				\begin{scope}
					\clip (17.6776695296637, -17.6776695296637) arc (-45:45:25) -- (17.6776695296637, 17.6776695296637) arc (225:45:10) -- (31.8198051533946, 31.8198051533946) arc (45:-45:45) --(31.8198051533946, -31.8198051533946) arc (315:135:10) -- cycle;
					
					\fill[teal!20] 
					(34.1621123082939, 29.2907849439976) arc (139.3900296688734:220.6099703311267:45) -- 
					(34.1621123082939, -75) -- 
					(34.1621123082939, 75) -- cycle;
				\end{scope}
				
				\begin{scope}
					\clip (17.6776695296637, -17.6776695296637) arc (-45:45:25);
					\fill[orange!90] 
					(34.1621123082939, 29.2907849439976) arc (139.3900296688734:220.6099703311267:45) -- cycle;
				\end{scope}
				
				\draw (17.6776695296637, -17.6776695296637) arc (-45:45:25);
				\draw (17.6776695296637, 17.6776695296637) arc (225:45:10);
				\draw (31.8198051533946, 31.8198051533946) arc (45:-45:45);
				\draw (31.8198051533946, -31.8198051533946) arc (315:135:10);
				
				\draw (53.0330085889911, 53.0330085889911) arc (45:-45:75);
				\draw (53.0330085889911, 53.0330085889911) arc (45:218.2226079814105:40);
				\draw (53.0330085889911, -53.0330085889911) arc (315:141.7773920185895:40);
				
				\draw[dashed] (34.1621123082939,75) -- (34.1621123082939,-75);
				
				\draw[densely dotted] (34.1621123082939, 63.6253160456164) arc (103.611333799743:135:40);
				\draw[densely dotted] (15.2912160275968, 53.0330085889911) arc (135:225:75);
				\draw[densely dotted] (15.2912160275968, -53.0330085889911) arc (225:256.388666200257:40);
				
				\draw[densely dotted] (34.1621123082939, 29.2907849439976) arc (139.3900296688734:220.6099703311267:45); 
				
				\draw (34.5,-71) node [anchor=west][inner sep=0.75pt]  [align=center] {$T_{\star}$};
				\draw (35,-35) node [anchor=south west][inner sep=0.75pt]  [align=center] {$\Omega$};
				\draw (69,-35) node [anchor=south west][inner sep=0.75pt]  [align=center] {$G$};
				\draw (26,-2) node [anchor=west][inner sep=0.75pt]  [align=center] {$\Omega^{\star}_\sH$};
				\draw (-0.5,-2) node [anchor=west][inner sep=0.75pt]  [align=center] {$G^{\star}_\sH$};
			\end{tikzpicture}
			\caption{An example where $\Omega^{\star}_\sH$ is not entirely contained in $\Omega$: the orange region belongs to $\Omega^{\star}_\sH$ but not to $\Omega$.}			
		\end{subfigure}
		\caption{An example in which our assumptions on $\Omega$ and $G$ are violated.}
		\label{fig:count-posreach-B}
	\end{figure}
	
	\begin{lemma}
		\label{lem:inclusion}
		Let~$\Omega \subseteq \R^n$ be a regular open, bounded set whose closure~$\overline{\Omega}$ has positive reach~$r_{\Omega} > 0$, let~$t \in (0,r_\Omega)$ in the definition of~$G= G_t$ in~\eqref{eq:defG}, and let~$\lambda_\star$ be the critical value defined in~\eqref{eq:lambda-crit}. Then, we have
		\begin{equation*}
			\label{eq:inclusion-Omega}
			\Omega^{\star}_\sH \subseteq \Omega.
		\end{equation*}
	\end{lemma}
	
	\begin{proof}
		Since~$\Omega^{\star}_\sH \neq \varnothing$ is an open set, it suffices to prove that~$\Omega^{\star}_\sH \subseteq \overline{\Omega}$, because then it follows that $\Omega^{\star}_\sH \subseteq \mathrm{int}(\overline{\Omega}) = \Omega$ since $\Omega$ is regular open.   
		
		Suppose, for contradiction, that there exists a point~$x \in \Omega^{\star}_\sH \setminus \overline{\Omega}$. We write~$x = \sigma_{\star}(\xi)$ with~$\xi \in \Omega_{\star} \subseteq \Omega$ and set~$\lambda_1 \coloneqq \xi_1$, so that $\xi \in T_{\lambda_1}$.
		
		Since~$\Omega$ is open, there exists~$\varepsilon>0$ such that
		\begin{equation}
			\label{eq:ximu-clostol1}
			\sigma_{\mu}(\xi) \in \Omega \quad \text{for } \mu \in (\lambda_1-\varepsilon,\lambda_1].
		\end{equation}
		We now define
		\begin{equation*}
			\lambda_2 \coloneqq \inf \left\{\lambda \in \R \mid \sigma_\mu(\xi) \in \Omega \text{ for all } \mu \in (\lambda,\lambda_1]\right\}\!.
		\end{equation*}
		By this definition, together with~\eqref{eq:ximu-clostol1} and the fact that~$x = \sigma_{\star}(\xi) \not\in \overline{\Omega}$, we obtain
		\begin{equation}
			\label{eq:lstar<l2}
			\lambda_\star < \lambda_2 < \lambda_1
		\end{equation}
		and
		\begin{equation}
			\label{eq:x-0-property}
			x_0 \coloneqq \sigma_{\lambda_2}(\xi) \in \partial \Omega \cap \Omega^{\lambda_2}_\sH.
		\end{equation}
		
		Moreover, since~$x_0 + s e_1 \in \Omega$ for small~$s>0$, it follows that~$e_1 \in \mathrm{Tan}(\overline \Omega,x_0)$. Let~$v$ be a unit vector such that~$v \in \mathrm{Nor}(\overline \Omega,x_0)$, which exists by point~\ref{it:Nor-nonemp} in Lemma~\ref{lem:property-posreach}. Then, Federer's inequality from point~\ref{it:Federer-ineq} in Lemma~\ref{lem:property-posreach} shows that
		\begin{equation*}
			v \cdot (y-x_0) \le \frac{|y-x_0|^2}{2 r_\Omega} \quad \text{for all } y \in \overline{\Omega}.
		\end{equation*}
		Consequently, for~$y \in \overline{\Omega} \setminus \{x_0\}$ and~$\tau \in (0,r_{\Omega})$, we have that
		\begin{align*}
			|x_0 + \tau v- y|^2 &= \tau^2 + |y-x_0|^2 - 2\tau v \cdot (y-x_0) \\       &\ge \tau^2 + \left(1- \frac{\tau}{r_\Omega}\right)|y-x_0|^2 > \tau^2.
		\end{align*}
		Thus,~$|x_0 + \tau v- y|> \tau = \dist(x_0+ \tau v,x_0)$, which implies that
		\begin{equation}
			\label{eq:distance-property-normal}
			\dist(x_0 + \tau v, \overline{\Omega}) = \tau \quad \text{for } \tau \in (0,r_{\Omega}).
		\end{equation}
		Furthermore, since~$e_1 \in \mathrm{Tan}(\overline \Omega,x_0)$ and~$v \in \mathrm{Nor}(\overline \Omega,x_0)$, we obtain from~\eqref{eq:def-Nor-cone} that~$v_1 = v \cdot e_1 \le 0$. Combining this with~\eqref{eq:x-0-property}, we deduce that
		\begin{equation}
			\label{eq:x-0-property-ray}
			x_0 + \tau v \in  H^{\lambda_2} \quad \text{for all } \tau>0.
		\end{equation}
		Since~$t< r_{\Omega}$ in the definition of~$G$,~\eqref{eq:distance-property-normal} and~\eqref{eq:x-0-property-ray} yield
		\begin{equation}
		\label{eq:x_0+tv-in1}
			x_0 + t v \in \partial G \cap H^{\lambda_2}.
		\end{equation}
		Moreover, since~$\Omega$ is open,~$\xi =\sigma_{\lambda_2}(x_0) \in \Omega$, and~$\sigma_{\lambda_2}(v)$ is again a unit vector, we also find that 
		\begin{equation*}
			\dist(\sigma_{\lambda_2}(x_0+ tv),\Omega)= \dist(\xi + t \sigma_{\lambda_2}(v),\Omega)<
			\abs*{\xi + t \sigma_{\lambda_2}(v)-\xi} = t
		\end{equation*}
		and therefore~$\sigma_{\lambda_2}(x_0+ tv) \in G$, which implies that 
		\begin{equation}
			\label{eq:x_0+tv-in2}
			x_0+ tv \in G^{\lambda_2}.
		\end{equation}
		Combining~\eqref{eq:x_0+tv-in1} and~\eqref{eq:x_0+tv-in2} gives~$x_0 + t v \in G^{\lambda_2} \cap H^{\lambda_2}= G^{\lambda_2}_\sH$ and~$x_0 +tv \in \partial G$. This is impossible by the definition of~$\lambda_\star$ since, from~\eqref{eq:lstar<l2}, we know that~$\lambda_2 > \lambda_\star$ and~$G$ is open. Hence, we have arrived at a contradiction, thereby concluding the proof.
	\end{proof}
	
	\begin{remark}
		In light of Lemma~\ref{lem:inclusion}, we also notice that~$\lambda_{\star}>-\infty$, and that the procedure cannot terminate after fully passing through~$\Omega$. In other words, it is impossible that~$\varnothing \neq \Omega = \Omega_{\star} \subseteq H_{\star}$, since in that case we would have~$\varnothing \neq \Omega^{\star}_\sH \subseteq \Omega \subseteq H_{\star}$ by Lemma~\ref{lem:inclusion}, leading to a contradiction with~$H_{\star} \cap H^{\star} = \varnothing$.
	\end{remark}
	
	Before we complete the proof of Theorem~\ref{th:symm-fractional}, we note the following simple but useful fact.
	\begin{lemma}
	\label{lem:regular-open-simple-lemma}
		If~$\Omega, \Omega' \subseteq \R^n$ are regular open sets with~$|\Omega|=|\Omega'|<+\infty$ and~$|\Omega \setminus \Omega'|=0$, then we have~$\Omega = \Omega'$.     
	\end{lemma}
	
	\begin{proof}
		We first note that~$|\Omega \setminus \overline{\Omega'}| \leq |\Omega \setminus \Omega'|= 0$, which implies that~$\Omega \setminus \overline{\Omega'}= \varnothing$, since~$\Omega \setminus \overline{\Omega'}$ is an open set. Thus, we have~$\Omega \subseteq \overline{\Omega'}$, and therefore also
		\begin{equation*}
			\Omega \subseteq \mathrm{int}(\overline{\Omega'}) = \Omega',
		\end{equation*}
		as~$\Omega$ is open and~$\Omega'$ is regular open. To show the inclusion~$\Omega' \subseteq \Omega$, we can argue in the same way after noting that also
		\begin{equation*}
			|\Omega' \setminus \Omega| = |\Omega'| -|\Omega \cap \Omega'|= |\Omega| -|\Omega \cap \Omega'|= |\Omega \setminus \Omega'| = 0.
		\end{equation*}
		Hence, the claim follows.  
	\end{proof}
	
	We are now in a position to prove Theorem~\ref{th:symm-fractional}.
	
	\begin{proof}[Proof of Theorem~\ref{th:symm-fractional}]
		To begin, we observe that~$\Omega$ can be decomposed, up to negligible sets, as
		\begin{equation}
			\label{eq:decomposition}
			\Omega = \left(\Omega \setminus \Omega^{\star} \right) \cup \left(\Omega \cap \Omega^{\star} \right)\qquad \text{with}\qquad 
			\Omega \cap \Omega^{\star} = \Omega_{\star} \!\cup \Omega^{\star}_\sH.
		\end{equation}
		To see the second identity in~\eqref{eq:decomposition}, we note that the inclusion $\Omega \cap \Omega^{\star} \subseteq \Omega_{\star} \!\cup \Omega^{\star}_\sH$ follows readily from the definition of these sets, while the inclusion $\Omega_{\star} \!\cup \Omega^{\star}_\sH \subseteq \Omega \cap \Omega^{\star}$ is a consequence of Lemma~\ref{lem:inclusion}. Moreover, we introduce the function
		\begin{equation}
		\label{eq:defv-fractional}
			v_{\star} \coloneqq u - u_{\star} \in H^s(\R^n),
		\end{equation}
		which is antisymmetric with respect to~$T_{\star}$, that is~$v_{\star}(\sigma_{\star}(x)) = - v_{\star}(x)$ for every~$x \in \R^n$. We claim that
		\begin{equation}
			\label{eq:wmp-weak}
			v_{\star} \ge 0 \quad \text{in } H^{\star}.
		\end{equation}
		This is clear if~$\Omega^{\star}_\sH = \varnothing$, since in this case~$\Omega \subseteq H^{\star}$ and therefore~$v_{\star} = u \ge 0$ in~$H^{\star}$. In the case where~$\Omega^{\star}_\sH \not = \varnothing$, it follows, since~$\Omega^{\star}_\sH \subseteq \Omega$ by Lemma~\ref{lem:inclusion}, from~\eqref{eq:mainprob-frac} that the function~$v_{\star}$ satisfies
		\begin{equation}
			\label{eq:eq-for-v}
			\begin{cases}
				\begin{aligned}
					\left(-\Delta\right)^s v_{\star} &= 0	&& \text{in } \Omega^{\star}_\sH, \\
					v_{\star} &\geq 0 						&& \text{in } H^{\star} \setminus \Omega^{\star}_\sH.
				\end{aligned}
			\end{cases}
		\end{equation}
		Notice that the second inequality in~\eqref{eq:eq-for-v} holds because, by~\eqref{eq:mainprob-frac} and~\eqref{eq:strict-posi},~$u \geq 0$ in~$\R^n$,~$u_{\star}=0$ in~$\R^n \setminus \Omega^{\star}$, and~$H^{\star} \setminus \Omega^{\star}_\sH \subseteq \R^n \setminus \Omega^{\star}$. By the weak maximum principle of Proposition~3.1 in~\cite{fj}, we then deduce that~$v_{\star} \geq 0$ in~$\Omega^{\star}_\sH$ and therefore~\eqref{eq:wmp-weak} also holds in this case.
		
		Next we note that, since~$x_\star \in \partial G \cap \partial G^\star$ in both~\ref{it:int-touch} and~\ref{it:non-transv-inter}, by~\eqref{eq:over-cond}, we have
		\begin{equation}
			\label{eq:equality-nonloc-deriv}
			\mathcal{N}_s u(x_\star) = c_\flat = \mathcal{N}_s u_{\star}(x_\star),
		\end{equation}
		where~$\mathcal{N}_s u$ is given by~\eqref{eq:expr-nonlocal-normal} and~$\mathcal{N}_s u_{\star}$ is coherently given by
		\begin{equation}
			\label{eq:expr-nonloc-der-lambda}
			\mathcal{N}_s u_{\star}(x) = -c_{n,s} \int_{\Omega^{\star}} \frac{u_{\star}(y)}{\abs*{x-y}^{n+2s}} \, dy \quad\text{for all } x \in \R^n \setminus \overline{\Omega^{\star}}.
		\end{equation}
		By~\eqref{eq:equality-nonloc-deriv}, we deduce, using a change of variables, that
		\begin{equation}
		\label{eq:defI}
			\begin{split}
				0 &= \int_{\Omega} \frac{u(y)}{\abs*{x_\star-y}^{n+2s}} \, dy - \int_{\Omega^{\star}} \frac{u_{\star}(y)}{\abs*{x_\star-y}^{n+2s}} \, dy \\
				& = \int_{\Omega} u(y) \left[\frac{1}{\abs*{x_\star-y}^{n+2s}} - \frac{1}{\abs*{x_\star-\sigma_{\star}(y)}^{n+2s}} \right] dy \eqqcolon I.
			\end{split}
		\end{equation}

		Now, we first consider~\ref{it:int-touch}, so we have~$x_\star \in H^{\star}$. This implies that
		\begin{equation}
			\label{basic-ineq}
			\abs*{x_\star-y} < \abs{x_\star-\sigma_{\star}(y)} \quad \text{for all } y \in H^{\star}.
		\end{equation}
		Using the decomposition in~\eqref{eq:decomposition}, we can now write~$I=I_1+I_2$ with
		\begin{equation}
			\label{eq:def-I_1}
			I_1  \coloneqq \int_{\Omega \setminus \Omega^{\star}} u(y) \left[\frac{1}{\abs*{x_\star-y}^{n+2s}} - \frac{1}{\abs*{x_\star-\sigma_{\star}(y)}^{n+2s}} \right] dy
		\end{equation}
		and 
		\begin{align*}
			I_2 & \coloneqq \int_{\Omega_{\star} \cup \Omega^{\star}_\sH}u(y) \left[\frac{1}{\abs*{x_\star-y}^{n+2s}} - \frac{1}{\abs*{x_\star-\sigma_{\star}(y)}^{n+2s}} \right] dy\\
			&= \int_{\Omega^{\star}_\sH} v_{\star}(y) \left[\frac{1}{\abs*{x_\star-y}^{n+2s}}-\frac{1}{\abs*{x_\star-\sigma_{\star}(y)}^{n+2s}}\right] dy.
		\end{align*}
		Here, in the second equality, we used the fact that~$\Omega^{\star}_\sH= \sigma_{\star}(\Omega_{\star})$, and that the kernel function
		\begin{equation*}
			y \mapsto \frac{1}{\abs*{x_\star-y}^{n+2s}}-\frac{1}{\abs*{x_\star-\sigma_{\star}(y)}^{n+2s}}
		\end{equation*}
		is odd with respect to the reflection~$\sigma_\star$. Since~$\Omega \setminus \Omega^{\star} \subseteq H^{\star}$ up to a set of measure zero and, by~\eqref{eq:strict-posi}, we have~$u>0$ in~$\Omega$, we deduce from~\eqref{basic-ineq}--\eqref{eq:def-I_1} that either~$|\Omega \setminus \Omega^{\star}|= 0$ or~$I_1>0$. Moreover, we have~$I_2 \ge 0$ by~\eqref{eq:wmp-weak} and~\eqref{basic-ineq}.
		
		As a result, we conclude from~\eqref{eq:defI} that~$|\Omega \setminus \Omega^{\star}|= 0$ and therefore~$\Omega = \Omega^\star$ by Lemma~\ref{lem:regular-open-simple-lemma}. So,~$\Omega$ is symmetric with respect to the reflection~$\sigma_{\star}$. \newline
		
		We then proceed to~\ref{it:non-transv-inter}. Since~$x_\star \in T_{\star} \cap \partial G \cap \partial G^{\star}$ is well detached from~$\Omega$ and~$\Omega^{\star}$, it follows that~$\mathcal{N}_s u$ and~$\mathcal{N}_s u_{\star}$ are of class~$C^1(B_r(x_\star))$ for some small~$r>0$. Moreover, given a direction~$\omega \in \sfera^{n-1}$, we have
		\begin{align*}
			\partial_{\omega} \,\mathcal{N}_s u(x) &= c_{n,s} \left(n+2s\right) \int_{\Omega} \frac{u(y)}{\abs*{x-y}^{n+2s+2}} \left(x-y\right) \cdot \omega \, dy, \\
			\partial_{\omega} \,\mathcal{N}_s u_{\star}(x) &= c_{n,s} \left(n+2s\right) \int_{\Omega^{\star}} \frac{u_{\star}(y)}{\abs*{x-y}^{n+2s+2}} \left(x-y\right) \cdot \omega \, dy \quad\text{for all } x \in B_r(x_\star).
		\end{align*}
		Because~$e_1 \in T_{x_\star} \partial G \cap T_{x_\star} \partial G^{\star}$,~$\mathcal{N}_s u$ is constant along~$\partial G$, and~$\mathcal{N}_s u_{\star}$ is constant along~$\partial G^{\star}$, we have
		\begin{equation*}
			\partial_{e_1} \,\mathcal{N}_s u(x_\star) = 0 = \partial_{e_1} \,\mathcal{N}_s u_{\star}(x_\star).
		\end{equation*}
		As a result, we deduce that
		\begin{equation}
		\label{eq:defJ}
			\begin{split}
				0 &= \int_{\Omega} \frac{u(y)}{\abs*{x_\star-y}^{n+2s+2}} \left(x_\star-y\right) \cdot e_1 \, dy - \int_{\Omega^{\star}} \frac{u_{\star}(y)}{\abs*{x_\star-y}^{n+2s+2}} \left(x_\star-y\right) \cdot e_1 \, dy \\
				&= \int_{\Omega} \frac{u(y)}{\abs*{x_\star-y}^{n+2s+2}} \left(\lambda_\star-y_1\right) dy - \int_{\Omega^{\star}} \frac{u_{\star}(y)}{\abs*{x_\star-y}^{n+2s+2}} \left(\lambda_\star-y_1\right) dy \eqqcolon J.
			\end{split}
		\end{equation}
		Since~$x_\star \in T_{\star}$, we also have
		\begin{equation*}
			\label{eq:norm-xstar-in-Tlstar}
			\abs{x_\star-\sigma_\star(y)}=\abs*{x_\star-y} \qquad \text{for all $y \in \R^n$.}
		\end{equation*}
		Thus, changing variables, we find that 
		\begin{align*}
			J &= \int_{\Omega} \frac{u(y)}{\abs*{x_\star-y}^{n+2s+2}} \left(\lambda_\star-y_1\right) dy + \int_{\Omega} \frac{u(y)}{\abs*{x_\star-\sigma_{\star}(y)}^{n+2s+2}} \left(\lambda_\star-y_1\right) dy \\
			&= 2 \int_{\Omega} \frac{u(y)}{\abs*{x_\star-y}^{n+2s+2}} \left(\lambda_\star-y_1\right) dy.
		\end{align*}
		Exploiting again the decomposition made in~\eqref{eq:decomposition}, we write~$J = J_1+J_2$ with
		\begin{equation}
			\label{eq:def-J-1}
			J_1  \coloneqq 2 \int_{\Omega \setminus \Omega^\star} \frac{u(y)}{\abs*{x_\star-y}^{n+2s+2}} \left(\lambda_\star-y_1\right)
			dy 
		\end{equation}
		and
		\begin{equation*}
			J_2  \coloneqq 2 \int_{\Omega_{\star} \cup \Omega^{\star}_\sH} \frac{u(y)}{\abs*{x_\star-y}^{n+2s+2}} \left(\lambda_\star-y_1\right)dy =2 \int_{\Omega^{\star}_\sH}v_{\star}(y) \,\frac{\lambda_\star-y_1}{\abs*{x_\star-y}^{n+2s+2}}\, dy.
		\end{equation*}
		Recall here that~$\Omega \setminus \Omega^{\star} \subseteq H^{\star}$ up to a set of measure zero, while~$\lambda_\star-y_1 >0$ for~$y \in H^{\star}$ and~$u>0$ in~$\Omega$, thanks to~\eqref{eq:strict-posi}. Thus, from~\eqref{eq:def-J-1}, we deduce that either~$|\Omega \setminus \Omega^{\star}|= 0$ or~$J_1>0$. Moreover, we have~$J_2 \ge 0$ by~\eqref{eq:wmp-weak} and since~$\Omega^{\star}_\sH \subseteq H^{\star}$.
		As a consequence,~\eqref{eq:defJ} implies that~$|\Omega \setminus \Omega^{\star}|= 0$, and therefore~$\Omega = \Omega^\star$ by Lemma~\ref{lem:regular-open-simple-lemma}. So also in this case we conclude that~$\Omega$ is symmetric with respect to the reflection~$\sigma_{\star}$.
		
		As noted at the beginning of this section, the direction~$e_1$ can be replaced by any other direction in the argument above. Then, it easily follows that~$\Omega$ must be a ball, thereby concluding the proof.
	\end{proof}
	
	\begin{remark}
		We emphasize that our proof relies only on the weak maximum principle for antisymmetric functions applied to the comparison function defined in~\eqref{eq:defv-fractional}, and not the strong maximum principle. This potentially allows the extension of Theorem~\ref{th:symm-fractional} to more general operators in place of~$(-\Delta)^s$ that do not enjoy the strong comparison principle.
	\end{remark}
	
	
	\section{Proof of the stability result for the torsion problem}
	\label{sec:proof-stabil-fractional}
	
	In the following we will continue to use the notation introduced in Section~\ref{sec:proof-sym-fractional} while applying the moving plane technique.
	
	First, we establish the subsequent result, which is needed to derive the uniform stability in one direction -- the content of Proposition~\ref{prop:unif-stab-each-dir} below.
	
	\begin{lemma}
		\label{lem:quant-est-1dir}
		Assume that~$n \geq 2$ or~$n=1$ and~$s \geq 1/2$.
		Suppose also that~$\Omega \subseteq \R^n$ is a regular open, bounded set whose closure has positive reach~$r_{\Omega} > 0$, and let~$t \in (0,r_{\Omega})$. Moreover, for a given direction~$\omega \in \sfera^{n-1}$, let~$\lambda_\star$ be the critical value defined as in~\eqref{eq:lambda-crit} with~$e_1$ replaced by~$\omega$, let~$T_{\star}$ be the associated critical hyperplane, and let~$\Omega^\star$ be the reflection of~$\Omega$ across~$T_{\star}$. Then, the unique solution~$u$ to~\eqref{eq:mainprob-frac} satisfies 
		\begin{equation}
			\label{eq:quant-est-1dir}
			\int_{\Omega \setminus \Omega^\star} \mathrm{dist}\!\left(x,T_{\star}\right) u(x) \, dx \leq \frac{\left(\mathrm{diam}(\Omega) + r_{\Omega}\right)^{n+2s+2}}{c_{n,s} \left(n+2s\right)} \left[\mathcal{N}_s u\right]_{\sGt}\!.
		\end{equation}
	\end{lemma}
	\begin{proof}
		Of course, after a rotation, we may assume~$\omega = e_1$. Thus, suppose that~$\lambda_\star$ is the critical value associated with the~$e_1$-direction. We define the function~$v_{\star}$ and the quantities~$I$,~$I_1$,~$I_2$,~$J$,~$J_1$, and~$J_2$ as in the proof of Theorem~\ref{th:symm-fractional} above, and observe that~\eqref{eq:wmp-weak} still holds, i.e., we have 
		\begin{equation}
			\label{eq:vnonneg-Omegalamb'}
			v_\star \geq 0 \quad\text{in } H^{\star}, 
		\end{equation}
		as this inequality only relies on the weak maximum principle and not on the overdetermined condition~\eqref{eq:over-cond}. We now first consider~\ref{it:int-touch}. Recalling~\eqref{eq:expr-nonloc-der-lambda}, a change of variables gives
		\begin{equation*}
			\mathcal{N}_s u_{\star}(x) = \mathcal{N}_s u(\sigma_\star(x)) \quad\text{for all } x \in \R^n \setminus \overline{\Omega^\star}.
		\end{equation*}
		Therefore, we can write
		\begin{equation}
			\label{eq:est-quantit-case1}
			\mathcal{N}_s u(\sigma_\star(x_\star)) - \mathcal{N}_s u(x_\star) = \mathcal{N}_s u_{\star}(x_\star) - \mathcal{N}_s u(x_\star) = c_{n,s} \, I = c_{n,s} \,(I_1+I_2).
		\end{equation}
		By~\eqref{basic-ineq} and~\eqref{eq:vnonneg-Omegalamb'}, we still have~$I_2 \geq 0$. We now aim to estimate~$I_1$, arguing similarly to the proof of Proposition~3.1 of~\cite{cfmnov}. Recalling that~$n \geq 2$ or~$n=1$ and~$s \geq 1/2$, by convexity, we have
		\begin{align*}
			\frac{1}{\abs*{x_\star-y}^{n+2s}}&-\frac{1}{\abs*{x_\star-\sigma_\star(y)}^{n+2s}} =\frac{1}{\abs*{x_\star-y}^{n+2s}}-\frac{1}{\abs*{\sigma_\star(x_\star)-y}^{n+2s}} \\
			&\quad= \frac{1}{\abs*{\sigma_\star(x_\star)-y}^{n+2s}} \left[\left(\frac{\abs{\sigma_\star(x_\star)-y}}{\abs*{x_\star-y}}\right)^{\! n+2s}-1\right] \\
			&\quad= \frac{1}{\abs*{\sigma_\star(x_\star)-y}^{n+2s}} \left[\left(1+\frac{4\left(\lambda_\star-x_\star^1\right)\left(\lambda_\star-y_1\right)}{\abs*{x_\star-y}^2}\right)^{\!\! \frac{n+2s}{2}}-1\right] \\
			&\quad\geq \frac{2\left(n+2s\right)\left(\lambda_\star-x_\star^1\right)\left(\lambda_\star-y_1\right)}{\abs*{\sigma_\star(x_\star)-y}^{n+2s} \,\abs*{x_\star-y}^2} \\
			&\quad\geq \frac{2\left(n+2s\right)\left(\lambda_\star-x_\star^1\right)\left(\lambda_\star-y_1\right)}{\left(\mathrm{diam}(\Omega) + r_{\Omega}\right)^{n+2s+2}} \geq 0 \quad \text{for } y \in \Omega \setminus \Omega^\star.
		\end{align*}
		Here~$x_\star^1$ denotes the first component of~$x_\star$. Recalling~\eqref{eq:strict-posi} and~\eqref{eq:def-I_1}, we deduce that
		\begin{equation*}
			I_1 \geq \frac{2\left(n+2s\right)}{\left(\mathrm{diam}(\Omega) + r_{\Omega}\right)^{n+2s+2}} \left(\lambda_\star-x_\star^1\right) \int_{\Omega \setminus \Omega^\star} u(y) \left(\lambda_\star-y_1\right) dy.
		\end{equation*}
		We also notice that~$\abs{\sigma_\star(x_\star)-x_\star} = 2 \left(\lambda_\star-x_\star^1\right)$. As a result, from~\eqref{eq:est-quantit-case1}, we conclude that
		\begin{equation*}
			\label{eq:tobededuced-1}
			\frac{\mathcal{N}_s u(\sigma_\star(x_\star)) - \mathcal{N}_s u(x_\star)}{\abs{\sigma_\star(x_\star)-x_\star}} \geq \frac{c_{n,s} \left(n+2s\right)}{\left(\mathrm{diam}(\Omega) + r_{\Omega}\right)^{n+2s+2}} \int_{\Omega \setminus \Omega^\star} u(y) \left(\lambda_\star-y_1\right) dy.
		\end{equation*}
		From the latter,~\eqref{eq:quant-est-1dir} follows at once in~\ref{it:int-touch}.
		
		We now proceed to~\ref{it:non-transv-inter}. We start by observing that~$(\sigma_\star(x)-y)\cdot e_1 = - (x-\sigma_\star(y))\cdot e_1$ for $x,y \in \R^n$, and therefore
		\begin{align}
			\notag
			\partial_{e_1} \,\mathcal{N}_s u(\sigma_\star(x)) &= c_{n,s} \left(n+2s\right) \int_{\Omega} \frac{u(y)}{\abs*{\sigma_\star(x) -y}^{n+2s+2}} \left(\sigma_\star(x)-y\right) \cdot e_1 \, dy \\
			\label{eq:comp-Nxlamb}
			&= - c_{n,s} \left(n+2s\right) \int_{\Omega} \frac{u(y)}{\abs*{\sigma_\star(x)-y}^{n+2s+2}} \left(x-\sigma_\star(y)\right) \cdot e_1 \, dy \\
			\notag
			&= - c_{n,s} \left(n+2s\right) \int_{\Omega^\star} \frac{u_{\star}(y)}{\abs*{\sigma_\star(x)-\sigma_\star(y)}^{n+2s+2}} \left(x-y\right) \cdot e_1 \, dy \\
			\notag
			&= - c_{n,s} \left(n+2s\right) \int_{\Omega^\star} \frac{u_{\star}(y)}{\abs*{x-y}^{n+2s+2}} \left(x-y\right) \cdot e_1 \, dy = - \partial_{e_1} \,\mathcal{N}_s u_{\star}(x),
		\end{align}
		for all~$x \in B_r(x_\star)$. Recall also that, in this case,~$x_\star \in T_{\star} \cap \partial G \cap \partial G^{\star}_\sH$ and thus~$\sigma_\star(x_\star)=x_\star$. Taking advantage of~\eqref{eq:comp-Nxlamb}, we can write
		\begin{equation*}
			2\partial_{e_1} \,\mathcal{N}_s u(x_\star) = \partial_{e_1} \,\mathcal{N}_s u(x_\star) - \partial_{e_1} \,\mathcal{N}_s u_{\star}(x_\star) = c_{n,s} \left(n+2s\right) J = c_{n,s} \left(n+2s\right) (J_1+J_2).
		\end{equation*}
		As in the proof of Theorem~\ref{th:symm-fractional},~\eqref{eq:vnonneg-Omegalamb'} yields that~$J_2 \geq 0$. Moreover, by the definition of~$J_1$ in~\eqref{eq:def-J-1}, we easily see that
		\begin{equation*}
			\label{eq:tobededuced-2}
			J_1 \geq \frac{2}{\left(\mathrm{diam}(\Omega) + r_{\Omega}\right)^{n+2s+2}} \int_{\Omega \setminus \Omega^\star} u(y) \left(\lambda_\star-y_1\right) dy.
		\end{equation*}
		Hence, we conclude that
		\begin{equation*}
			\partial_{e_1} \,\mathcal{N}_s u(x_\star) \geq \frac{c_{n,s} \left(n+2s\right)}{\left(\mathrm{diam}(\Omega) + r_{\Omega}\right)^{n+2s+2}} \int_{\Omega \setminus \Omega^\star} u(y) \left(\lambda_\star-y_1\right) dy,
		\end{equation*}
		from which inequality~\eqref{eq:quant-est-1dir} also follows.
	\end{proof}
	
	We now plan to derive the uniform stability in each direction, i.e., a quantitative estimate for~$\abs*{\Omega \triangle \Omega^\star}$, from Lemma~\ref{lem:quant-est-1dir}. This estimate is essentially contained in~\eqref{eq:quant-est-1dir} through arguments developed in~\cite{cfmnov}, provided one can get rid of~$u$ in the integral on the left-hand side of~\eqref{eq:quant-est-1dir}. This can, of course, be achieved by employing a quantitative version of Hopf's lemma to obtain a lower bound for~$u$.
	
	Once the stability estimate in one direction is obtained, the full stability result follows via a well-established procedure introduced in~\cite{cfmnov}.
	
	The details of the implementation in our situation are as follows.
	
	\begin{proposition}
		\label{prop:unif-stab-each-dir}
		Assume that~$n \geq 2$ or~$n=1$ and~$s \geq 1/2$.
		Moreover, suppose that~$\Omega \subseteq \R^n$ is a regular open, bounded set whose closure has positive reach~$r_{\Omega} > 0$, and let $t \in (0,r_{\Omega})$. Finally, assume that $\Omega$ satisfies the uniform interior sphere condition with radius~$\mathfrak{r}_\Omega>0$. For a given direction~$\omega \in \sfera^{n-1}$, let~$\lambda_\star$ be the critical value defined as in~\eqref{eq:lambda-crit} with~$e_1$ replaced by~$\omega$, let~$T_{\star}$ be the associated critical hyperplane, and let~$\Omega^\star$ be the reflection of~$\Omega$ across~$T_{\star}$. Then, the unique solution~$u$ to~\eqref{eq:mainprob-frac} satisfies
		\begin{equation*}
			\abs*{\Omega \triangle \Omega^\star} \leq C_\sharp \left[\mathcal{N}_s u\right]_{\sGt}^{\frac{1}{2+s}} \!,
		\end{equation*}
		with
		\begin{gather}
			\label{eq:def-Csharp}
			C_\sharp \coloneqq 2 \left(s+2\right) C_1^{\frac{1}{2+s}} \left(\frac{C_2}{s+1}\right)^{\!\frac{1+s}{2+s}} \!, \\
			\notag
			C_1 \coloneqq \frac{\left(\mathrm{diam}(\Omega) + r_{\Omega}\right)^{n+2s+2}}{c_{n,s} \,\gamma_{n,s} \, \mathfrak{r}_\Omega^s \left(n+2s\right)}, \quad C_2 \coloneqq \mathrm{diam}(\Omega)^{n-1}+\frac{n \,\abs*{B_1} \,\mathrm{diam}(\Omega)^{n}}{2^{n-1} \,\mathfrak{r}_{\Omega}}.
		\end{gather}
	\end{proposition}
	\begin{proof}
		Since~$\Omega$ is of class~$C^1$ and satisfies the uniform interior sphere condition with radius~$\mathfrak{r}_\Omega>0$, by Lemma~5.3 in~\cite{cir-pol-fractionaltorsion}, we have
		\begin{equation}
			\label{eq:est-u-below}
			u(x) \geq \gamma_{n,s} \, \mathfrak{r}_\Omega^s \,\mathrm{dist}\left(x,\partial\Omega\right)^s \quad\text{for all } x \in \Omega,
		\end{equation}
		where~$\gamma_{n,s}$ is defined in~\eqref{eq:gamma-ns}.
		
		Possibly after a rotation and a translation, we can suppose that~$\omega = e_1$ and~$\lambda_\star=0$. Thus,~\eqref{eq:quant-est-1dir} reads as
		\begin{equation*}
			\int_{\Omega \setminus \Omega^\star} -x_1 \, u(x) \, dx \leq \frac{\left(\mathrm{diam}(\Omega) + r_{\Omega}\right)^{n+2s+2}}{c_{n,s} \left(n+2s\right)} \left[\mathcal{N}_s u\right]_{\sGt}\!.
		\end{equation*}
		From now on, we argue as in the proof of Proposition~4.4 in~\cite{dptv-parallel}. The latter inequality, together with~\eqref{eq:est-u-below}, yields 		
		\begin{equation}
			\label{eq:est1-stab-1dir}
			\int_{\Omega \setminus \Omega^\star} -x_1 \,\mathrm{dist}\!\left(x,\partial\Omega\right)^s dx \leq \frac{\left(\mathrm{diam}(\Omega) + r_{\Omega}\right)^{n+2s+2}}{c_{n,s} \,\gamma_{n,s} \, \mathfrak{r}_\Omega^s \left(n+2s\right)} \left[\mathcal{N}_s u\right]_{\sGt} \eqqcolon C_1 \left[\mathcal{N}_s u\right]_{\sGt}\!.
		\end{equation}
		Let~$\gamma>0$ to be determined soon. Chebyshev’s inequality and~\eqref{eq:est1-stab-1dir} entail that
		\begin{align}
			\notag
			\abs*{\left\{x \in \Omega \setminus \Omega^\star \mid -x_1 \,\mathrm{dist}\!\left(x,\partial\Omega\right)^s > \gamma\right\}} &\leq \frac{1}{\gamma} \int_{\Omega \setminus \Omega^\star} -x_1 \,\mathrm{dist}\!\left(x,\partial\Omega\right)^s dx \\
			\label{eq:tobecoupled}
			&\leq \frac{C_1}{\gamma} \left[\mathcal{N}_s u\right]_{\sGt}\!.
		\end{align}
		In addition, we also have
		\begin{align*}
			&\abs*{\left\{x \in \Omega \setminus \Omega^\star \mid -x_1 \,\mathrm{dist}\!\left(x,\partial\Omega\right)^s \leq \gamma\right\}} \\
			&\quad= \abs*{\left\{x \in \Omega \setminus \Omega^\star \mid -x_1 \,\mathrm{dist}\!\left(x,\partial\Omega\right)^s \leq \gamma, -x_1 < \gamma^{\frac{1}{1+s}}\right\}} \\
			&\qquad+ \abs*{\left\{x \in \Omega \setminus \Omega^\star \mid -x_1 \,\mathrm{dist}\!\left(x,\partial\Omega\right)^s \leq \gamma, -x_1 \geq \gamma^{\frac{1}{1+s}}\right\}} \\
			&\quad\leq \abs*{\left\{x \in \Omega \cap H^{\star} \mid -x_1 < \gamma^{\frac{1}{1+s}}\right\}} + \abs*{\left\{x \in \Omega \mid \mathrm{dist}\!\left(x,\partial\Omega\right) \leq \gamma^{\frac{1}{1+s}}\right\}},
		\end{align*}
		and
		\begin{equation*}
			\abs*{\left\{x \in \Omega \cap H^{\star} \mid -x_1 < \gamma^{\frac{1}{1+s}}\right\}} \leq \mathrm{diam}(\Omega)^{n-1} \,\gamma^{\frac{1}{1+s}},
		\end{equation*}
		trivially, while
		\begin{equation}
			\label{eq:meas-tube}
			\abs*{\left\{x \in \Omega \mid \mathrm{dist}\!\left(x,\partial\Omega\right) \leq \gamma^{\frac{1}{1+s}}\right\}} \leq \frac{2n \,\abs*{\Omega}}{\mathfrak{r}_{\Omega}} \,\gamma^{\frac{1}{1+s}} \leq \frac{n \,\abs*{B_1} \,\mathrm{diam}(\Omega)^{n}}{2^{n-1} \,\mathfrak{r}_{\Omega}} \,\gamma^{\frac{1}{1+s}},
		\end{equation}
		by Equation~(1.10) in~\cite{tamanini} and Lemma~5.2 in~\cite{cir-pol-fractionaltorsion}. Here, we also used the isodiametric inequality for the last estimate. Hence, it follows that
		\begin{align*}
			\abs*{\left\{x \in \Omega \setminus \Omega^\star \mid -x_1 \,\mathrm{dist}\!\left(x,\partial\Omega\right)^s \leq \gamma\right\}} &\leq \left(\mathrm{diam}(\Omega)^{n-1}+\frac{n \,\abs*{B_1} \,\mathrm{diam}(\Omega)^{n}}{2^{n-1} \,\mathfrak{r}_{\Omega}}\right) \gamma^{\frac{1}{1+s}} \\
			&\eqqcolon C_2 \,\gamma^{\frac{1}{1+s}}.
		\end{align*}
		Coupling with~\eqref{eq:tobecoupled}, we get
		\begin{equation*}
			\abs*{\Omega \setminus \Omega^\star} \leq \frac{C_1}{\gamma} \left[\mathcal{N}_s u\right]_{\sGt} + C_2\,\gamma^{\frac{1}{1+s}}.
		\end{equation*}
		Minimizing over~$\gamma$, we conclude that
		\begin{equation*}
			\abs*{\Omega \setminus \Omega^\star} \leq \left(s+2\right) C_1^{\frac{1}{2+s}} \left(\frac{C_2}{s+1}\right)^{\!\frac{1+s}{2+s}} \left[\mathcal{N}_s u\right]_{\sGt}^{\frac{1}{2+s}}\!,
		\end{equation*}
		which, in turn, implies
		\begin{equation*}
			\abs*{\Omega \triangle \Omega^\star} = 2\,\abs*{\Omega \setminus \Omega^\star} \leq C_\sharp \left[\mathcal{N}_s u\right]_{\sGt}^{\frac{1}{2+s}} \!,
		\end{equation*}
		where~$C_\sharp>0$ is defined in~\eqref{eq:def-Csharp}.
	\end{proof}
	
	We are now in a position to prove that, for any direction,~$\lambda_\star$ must be close to the center of approximate symmetry, defined as the intersection of the critical hyperplanes associated with~$e_1,\dots,e_n$. This will be achieved by reasoning along the lines of Lemma~4.1 in~\cite{cfmnov}.
	
	\begin{proposition}
		\label{prop:est-abslstar}
		Assume that~$n \geq 2$ or~$n=1$ and~$s \geq 1/2$.
		Moreover, let~$\Omega \subseteq \R^n$ be a regular open, bounded set whose closure has positive reach~$r_{\Omega} > 0$, and let~$t \in (0,r_{\Omega})$. Assume that~$\Omega$ satisfies the uniform interior sphere condition with radius~$\mathfrak{r}_\Omega>0$, and that the unique solution~$u$ of~\eqref{eq:mainprob-frac} satisfies
		\begin{equation}
			\label{eq:defi-small}
			\left[\mathcal{N}_s u\right]_{\sGt}^{\frac{1}{2+s}} \leq \gamma_0 \coloneqq \min\left\{\frac{1}{4},\frac{1}{n}\right\}\frac{\mathfrak{r}_\Omega^n \,\abs*{B_1}}{C_\sharp},
		\end{equation}
		where~$C_\sharp>0$ is defined in~\eqref{eq:def-Csharp}.  Moreover, denote with~$T_i$ the critical hyperplane associated with the direction~$e_i$ and suppose that~$T_i = \{x \in \R^n \mid x_i=0\}$, for every~$i=1,\dots,n$. Finally, for a direction~$\omega \in \sfera^{n-1}$, let~$\lambda_\star$ be the critical value given in~\eqref{eq:lambda-crit} with~$e_1$ replaced by~$\omega$, and let~$T_{\star}$ be the associated critical hyperplane. Then, we have
		\begin{equation*}
			\abs*{\lambda_\star} \leq C_\flat \left[\mathcal{N}_s u\right]_{\sGt}^{\frac{1}{2+s}}  \qquad \text{with}\qquad 	C_\flat \coloneqq 4 \left(n+3\right) \frac{C_\sharp \,\mathrm{diam}(\Omega)}{\mathfrak{r}_{\Omega}^n \,\abs*{B_1}}.
		\end{equation*}
	\end{proposition}
	\begin{proof}
		Define~$\widetilde{\Omega} \coloneqq \left\{-x \mid x \in \Omega\right\}$. Since~$\widetilde{\Omega}$ can be obtained by symmetrizing~$\Omega$ with respect to~$T_i$ for~$i=1,\dots,n$, we deduce that
		\begin{equation}
			\label{eq:est-Omega-Omega0}
			\abs*{\Omega \triangle {\widetilde \Omega}} \leq n C_\sharp \left[\mathcal{N}_s u\right]_{\sGt}^{\frac{1}{2+s}} \!,
		\end{equation}
		by applying Proposition~\ref{prop:unif-stab-each-dir} with respect to the coordinate directions. See also the proof of Lemma~6.1 in~\cite{cir-pol-fractionaltorsion} for further details about the estimate in~\eqref{eq:est-Omega-Omega0}.
		
		Fix~$\omega \in \sfera^{n-1}$ and assume that~$\lambda_\star>0$, as the case~$\lambda_\star<0$ is analogous. We observe that
		\begin{equation}
			\label{eq:est-Lambda}
			\Lambda_\star \coloneqq \sup\left\{x \cdot \omega \mid x \in \Omega \right\} \leq \mathrm{diam}(\Omega).
		\end{equation}
		Indeed, if~$\Lambda_\star > \mathrm{diam}(\Omega)$ and~$x \in \Omega$, then~$x \cdot \omega \geq \Lambda_\star - \mathrm{diam}(\Omega) > 0$. Consequently, we have~$\abs*{\Omega \triangle \widetilde{\Omega}}=2\,\abs*{\Omega}$ which, combined with~\eqref{eq:defi-small} and~\eqref{eq:est-Omega-Omega0}, would lead to
		\begin{equation*}
			0 < 2\,\abs*{\Omega} \leq n C_\sharp \,\gamma_0 \leq \mathfrak{r}_\Omega^n \,\abs*{B_1} \leq \abs*{\Omega},
		\end{equation*}
		a contradiction. This proves the validity of~\eqref{eq:est-Lambda}.
		
		We continue to use the notation introduced at the beginning of Section~\ref{sec:proof-sym-fractional}, with the direction~$e_1$ replaced by~$\omega$. By Proposition~\ref{prop:unif-stab-each-dir}, it follows that
		\begin{equation}
			\label{eq:est-Omega-Omega'}
			\abs*{\Omega \triangle \Omega^\star} \leq C_\sharp \left[\mathcal{N}_s u\right]_{\sGt}^{\frac{1}{2+s}} \!.
		\end{equation}
		Recalling the decomposition~$\Omega = (\Omega \setminus \Omega^\star) \cup \Omega_\star \cup \Omega^\star_\sH$,  we deduce from~\eqref{eq:est-Omega-Omega'} that 
		\begin{equation}
			\label{eq:est-below-1}
			\abs*{\Omega_{\star}}= \frac{\abs*{\Omega_\star \cup \Omega^\star_\sH}}{2}=
			\frac{\abs*{\Omega} -\abs*{\Omega \setminus \Omega^\star}}{2} \geq \frac{\abs*{\Omega}}{2} - \abs*{\Omega \triangle \Omega^\star} \geq \frac{\abs*{\Omega}}{2} - C_\sharp \left[\mathcal{N}_s u\right]_{\sGt}^{\frac{1}{2+s}} \!.
		\end{equation}
		Moreover, setting
		\begin{equation}
			\label{eq:def-H-widetilde}
			\widetilde{H}_{\star} \coloneqq \left\{-x \mid x \in H_{\star}\right\}= \left\{x \mid x \cdot \omega < -\lambda_\star \right\}\!,  
		\end{equation}
		we find that~\eqref{eq:est-Omega-Omega0} and~\eqref{eq:est-below-1} yield
		\begin{equation}
			\label{eq:est-below-2}
			\abs*{\Omega \cap \widetilde{H}_{\star}} = \abs*{\widetilde{\Omega} \cap H_{\star}} \geq \abs*{\Omega_{\star}} - \abs*{\Omega \triangle \widetilde{\Omega}} \geq \frac{\abs*{\Omega}}{2} - \left(n+1\right) C_\sharp \left[\mathcal{N}_s u\right]_{\sGt}^{\frac{1}{2+s}} \!.
		\end{equation}
		Exploiting~\eqref{eq:est-below-1},~\eqref{eq:def-H-widetilde}, and~\eqref{eq:est-below-2}, we see that
		\begin{equation}
			\label{eq:est-strip}
			\abs*{\left\{x \in \Omega \mid -\lambda_\star \leq x \cdot \omega \leq \lambda_\star \right\}} = \abs*{\Omega} - \abs*{\Omega_{\star}} - \abs*{\Omega  \cap \widetilde{H}_{\star}} \leq \left(n+2\right) C_\sharp \left[\mathcal{N}_s u\right]_{\sGt}^{\frac{1}{2+s}} \!.
		\end{equation}
		Since this strip is the image of~$\left\{\lambda_\star \leq x \cdot \omega \leq 3\lambda_\star\right\}$ through the reflection across~$T_{\star}$, we have
		\begin{align}
			\notag
			\abs*{\left\{x \in \Omega \mid \lambda_\star \leq x \cdot \omega \leq 3\lambda_\star \right\}} &= \abs*{\left\{x \in \Omega^\star \mid \abs*{x \cdot \omega} \leq \lambda_\star \right\}} \\
			\label{eq:est-m1}
			&\leq \abs*{\left\{x \in \Omega \mid \abs*{x \cdot \omega} \leq \lambda_\star \right\}} + \abs*{\Omega \triangle \Omega^\star} \\
			\notag
			&\leq \left(n+3\right) C_\sharp \left[\mathcal{N}_s u\right]_{\sGt}^{\frac{1}{2+s}} \!,
		\end{align}
		where we used~\eqref{eq:est-Omega-Omega'} and~\eqref{eq:est-strip}.
		
		We now partition~$\Omega$ into strips of width~$2\lambda_\star$, and we define
		\begin{equation*}
			m_k \coloneqq \abs*{\left\{x \in \Omega \mid \left(2k-1\right)\lambda_\star \leq x \cdot \omega \leq \left(2k+1\right)\lambda_\star \right\}} \quad \text{for } k \geq 1.
		\end{equation*}
		We note that, by definition of the critical value~$\lambda_\star$, a parallel translation by the vector~$(\mu'-\mu)\omega$ maps the set~$\Omega \cap T_{\mu}$ into the set~$\Omega \cap T_{\mu'}$, for every~$\lambda_\star \leq \mu' \leq \mu$. Therefore, we deduce that the sequence~$m_k$ is decreasing. Consequently, using~\eqref{eq:est-m1}, we find that
		\begin{equation*}
			m_k \leq m_1 \leq \left(n+3\right) C_\sharp \left[\mathcal{N}_s u\right]_{\sGt}^{\frac{1}{2+s}} \quad \text{for all } k \geq 1.
		\end{equation*}
		Now, let~$k_0 \in \N$ be the smallest natural number such that~$\left(2 k_0+1\right) \lambda_\star \geq \Lambda_\star$, hence
		\begin{equation*}
			\abs*{\Omega_{\star}} \leq \abs*{\Omega \cap \left\{\lambda_\star \leq x \cdot \omega \leq \Lambda_\star\right\}} \leq \sum_{k=1}^{k_0} m_k \leq \frac{1}{2} \left(\frac{\Lambda_\star}{\lambda_\star}+1\right) \left(n+3\right) C_\sharp \left[\mathcal{N}_s u\right]_{\sGt}^{\frac{1}{2+s}} \!.
		\end{equation*}
		Exploiting~\eqref{eq:est-Lambda}, we conclude that
		\begin{equation*}
			\abs*{\Omega_{\star}} \,\lambda_\star \leq \left(n+3\right) C_\sharp \,\mathrm{diam}(\Omega) \left[\mathcal{N}_s u\right]_{\sGt}^{\frac{1}{2+s}} \!.
		\end{equation*}
		Finally, by~\eqref{eq:defi-small} and~\eqref{eq:est-below-1}, we have
		\begin{equation*}
			\abs*{\Omega_{\star}} \geq \frac{\abs*{\Omega}}{2} - C_\sharp \left[\mathcal{N}_s u\right]_{\sGt}^{\frac{1}{2+s}} \geq \frac{\mathfrak{r}_{\Omega}^n \,\abs*{B_1}}{2} - C_\sharp \gamma_0 \ge  \frac{\mathfrak{r}_{\Omega}^n \,\abs*{B_1}}{4},
		\end{equation*}
		which yield the desired conclusion.
	\end{proof}
	
	We can finally prove the stability result for the fractional torsion problem.
	
	\begin{proof}[Proof of Theorem~\ref{th:stab-fractional}]
		Up to a translation and a rotation, we may assume that~$T_i = \{x \in \R^n \mid x_i=0\}$, for every~$i=1,\dots,n$. Moreover, we can suppose that~\eqref{eq:defi-small} is in force. Otherwise, we trivially have
		\begin{equation*}
			\rho(\Omega) \leq \mathrm{diam}(\Omega) \leq \frac{\mathrm{diam}(\Omega)}{\gamma_0} \left[\mathcal{N}_s u\right]_{\sGt}^{\frac{1}{2+s}} \!.
		\end{equation*}
		
		Define now
		\begin{equation*}
			r \coloneqq \min_{x \in \partial\Omega} \,\abs*{x} \quad\text{and}\quad R \coloneqq \max_{x \in \partial\Omega} \,\abs*{x},
		\end{equation*}
		and let~$x,y \in\partial\Omega$ be such that~$\abs*{x}=r$ and~$\abs*{y}=R$. Notice that, if~$x=y$, then~$\Omega$ is a ball and the result is obviously true. Therefore, we may assume~$x \neq y$ and define
		\begin{equation*}
			\omega \coloneqq \frac{y-x}{\abs*{y-x}} \in \sfera^{n-1}.
		\end{equation*}
		
		Let~$T_{\star}$ be the corresponding critical hyperplane. The moving plane method entails that
		\begin{equation}
			\label{eq:dist-x-y}
			\mathrm{dist}\!\left(x,T_{\star}\right) \geq \mathrm{dist}\!\left(y,T_{\star}\right)\!.
		\end{equation}
		Indeed, since~$x=y-\tau \omega$ with~$\tau=\abs*{y-x}$, the critical position can be reached at most when~$\sigma_\star(y)$ is tangent to~$x$, which corresponds to the equality case in~\eqref{eq:dist-x-y}, while in all the other cases strict inequality holds. See, for instance, Proposition~7 in~\cite{abr} for more details. Therefore, we get
		\begin{equation*}
			R-r = \abs*{y}-\abs*{x} \leq 2 \,\mathrm{dist}\!\left(0,T_{\star}\right) = 2\,\abs*{\lambda_\star}.
		\end{equation*}
		Hence, by applying Proposition~\ref{prop:est-abslstar}, we deduce that
		\begin{equation*}
			\rho(\Omega) \leq R-r \leq 2 C_\flat \left[\mathcal{N}_s u\right]_{\sGt}^{\frac{1}{2+s}} \!.
		\end{equation*}
		The conclusion follows with
		\begin{equation*}
			C \coloneqq \max\left\{2 C_\flat, \frac{\mathrm{diam}(\Omega)}{\gamma_0} \right\} = 2 C_\flat.
		\end{equation*}
	\end{proof}
	
	
	\section{Generalization of the symmetry and stability results}
	\label{sec:gen-sym-frac}
	
	This section is devoted to the proofs of Theorems~\ref{th:sym-frac-gen} and~\ref{th:stab-fractional-gen}. As mentioned in the Introduction, the arguments follow the approach of Sections~\ref{sec:proof-sym-fractional} and~\ref{sec:proof-stabil-fractional}, respectively, with technical variations already introduced in~\cite{dptv-parallel,fj}. For completeness, we reproduce the necessary arguments here, emphasizing only the modifications required to conclude the proofs in the generalized setting, as compared to those of Theorems~\ref{th:symm-fractional} and~\ref{th:stab-fractional}.
	
	We begin with the symmetry result.
	
	\begin{proof}[Proof of Theorem~\ref{th:sym-frac-gen}]
		We adopt the notation introduced in Section~\ref{sec:proof-sym-fractional}.
		Since, by assumption~$u>0$ in~$\Omega$, we only need to show~\eqref{eq:wmp-weak}, i.e.,
		\begin{equation}
			\label{eq:wmp-weak-generalization}
			v_\star \ge 0 \qquad \text{in }H^\star 
		\end{equation}
		to conclude in the same way as in the proof of Theorem~\ref{th:symm-fractional}. 
		
		In order to prove~\eqref{eq:wmp-weak-generalization}, we introduce, for any~$\lambda \in \R$, the function
		\begin{equation*}
			v_\lambda \coloneqq u - u_\lambda,
		\end{equation*}
		so that~$v_\star = v_{\lambda_\star}$. By the definition in~\eqref{eq:lambda-crit} and Lemma~\ref{lem:inclusion}, it follows that~$\Omega^\lambda_\sH \subseteq \Omega$ for all~$\lambda \in [\lambda_\star,\Lambda)$. Thus,~$v_\lambda \in H^s(\R^n)$ is an antisymmetric solution to
		\begin{equation}
			\label{eq:eqforv-gen}
			\begin{cases}
				\begin{aligned}
					& \!\left(-\Delta\right)^s v_\lambda + c_\lambda \, v_\lambda = 0	&& \text{in } \Omega^\lambda_\sH, \\
					& v_\lambda = u			&& \text{in } \left(\Omega \cap H^\lambda \right) \setminus \Omega^\lambda_\sH, \\
					& v_\lambda = 0			&& \text{in } H^\lambda \setminus \Omega,
				\end{aligned}
			\end{cases}
		\end{equation}
		where
		\begin{equation*}
			c_\lambda(x) \coloneqq
			\begin{cases}
				\begin{aligned}
					& - \frac{f(u(x))-f(u(\sigma_\lambda(x)))}{u(x) - u(\sigma_\lambda(x))}	&& \text{if } u(x) \neq u(\sigma_\lambda(x)), \\
					& 0		&& \text{if } u(x) = u(\sigma_\lambda(x)).
				\end{aligned}
			\end{cases}
		\end{equation*}
		Moreover, since~$u \in L^\infty(\R^n)$ and~\eqref{eq:ass-f-frac} is in force, we deduce that
		\begin{equation}
			\label{eq:unif-bound-clam}
			c_\lambda \in L^\infty(\Omega^\lambda_\sH) \quad\text{with}\quad \norma*{c_\lambda}_{L^\infty(\Omega^\lambda_\sH)} \leq \norma*{f}_{C^{0,1}\left([0,\norma*{u}_{L^\infty(\R^n)}]\right)} \quad\text{for all } \lambda \in [\lambda_\star,\Lambda).
		\end{equation}
		Let~$\lambda_1(\Omega^\lambda_\sH)$ denote the first Dirichlet eigenvalue of~$\left(-\Delta\right)^s$ in~$\Omega^\lambda_\sH$. Then, by Theorem~1.1 in~\cite{yy-eigenv}, it follows that
		\begin{equation*}
			\lambda_1(\Omega^\lambda_\sH) \geq C \,\abs*{\Omega^\lambda_\sH}^{-\frac{2s}{n}} \to +\infty \quad\text{as } \lambda \to \Lambda^{-}.
		\end{equation*}
		Taking advantage of this and of the uniform bound in~\eqref{eq:unif-bound-clam},
		the weak maximum principle of Proposition~3.1 in~\cite{fj} implies that there exists a~$\lambda_\flat \in [\lambda_\star,\Lambda)$ such that
		\begin{equation*}
			v_\lambda \geq 0 \quad\text{in } \Omega^\lambda_\sH \;\,\text{for all } \lambda \in [\lambda_\flat,\Lambda)
		\end{equation*}
		and, by~\eqref{eq:eqforv-gen}, also that
		\begin{equation}
			\label{eq:vlambda-nonneg-H'}
			v_\lambda \geq 0 \quad\text{in } H^\lambda \;\,\text{for all } \lambda \in [\lambda_\flat,\Lambda).
		\end{equation}
		
		Now, if~$x_0 \in \Omega \cap H^\lambda$ and choosing~$\lambda_\flat$ sufficiently close to~$\Lambda$, we have~$\sigma_\lambda(x_0) \not\in \Omega$ for all~$\lambda \in [\lambda_\flat,\Lambda)$. Thus, possibly taking a~$\lambda_\flat$ closer to~$\Lambda$, we obtain
		\begin{equation*}
			v_\lambda(x_0) = u(x_0)>0 \quad\text{for all } \lambda \in [\lambda_\flat,\Lambda).
		\end{equation*}
		Hence, from~\eqref{eq:vlambda-nonneg-H'} and the strong maximum principle of Corollary~3.4 in~\cite{fj}, we conclude that
		\begin{equation*}
			v_\lambda > 0 \quad\text{in } \Omega^\lambda_\sH \;\,\text{for all } \lambda \in [\lambda_\flat,\Lambda).
		\end{equation*}
		As a consequence, the quantity
		\begin{equation}
			\label{eq:def-mu-star}
			\mu_\star \coloneqq \inf \left\{\mu \in [\lambda_\star,\Lambda) \mid v_\lambda>0 \text{ in } \Omega^\lambda_\sH \text{ for all } \lambda \in [\mu,\Lambda)\right\}
		\end{equation}
		is well defined. We aim to show that~$\mu_\star=\lambda_\star$. This, together with~\eqref{eq:eqforv-gen}, clearly implies the validity of~\eqref{eq:wmp-weak-generalization}.
		
		We proceed by contradiction and assume that~$\mu_\star>\lambda_\star$. By continuity, we have~$v_{\mu_\star} \geq 0$ in~$\Omega^{\mu_\star}_\sH$ and, by~\eqref{eq:eqforv-gen}, also in~$H^{\mu_\star}$. Since~$\mu_\star>\lambda_\star$, the previous argument, together with Corollary~3.4 in~\cite{fj}, entails that~$v_{\mu_\star} > 0$ in~$\Omega^{\mu_\star}_\sH$. In light of~\eqref{eq:def-mu-star}, it follows that
		\begin{equation*}
			\left\{v_{\mu_\star-\varepsilon} \leq 0\right\} \cap \Omega^{\mu_\star-\varepsilon}_\sH \neq \varnothing \quad\text{for all } \varepsilon \in (0,\mu_\star-\lambda_\star).
		\end{equation*}
		Let~$A_\varepsilon \subseteq \Omega^{\mu_\star-\varepsilon}_\sH$ be an open set such that~$\left\{v_{\mu_\star-\varepsilon} \leq 0\right\} \cap \Omega^{\mu_\star-\varepsilon}_\sH \subseteq A_\varepsilon$. Taking~$\varepsilon>0$ sufficiently small, we can choose~$A_\varepsilon$ so that~$\abs*{A_\varepsilon}$ is arbitrarily small. Therefore, applying Proposition~3.1 and Corollary~3.4 in~\cite{fj}, we conclude that
		\begin{equation*}
			v_{\mu_\star-\varepsilon} > 0 \quad\text{in } A_\varepsilon.
		\end{equation*}
		This leads to a contradiction with the definition of~$\mu_\star$ in~\eqref{eq:def-mu-star}.
		
		Hence~\eqref{eq:wmp-weak-generalization} holds and the remainder of the proof proceeds exactly as in the proof of Theorem~\ref{th:symm-fractional}.
	\end{proof}
	
	We now turn to the stability result.
	
	\begin{proof}[Proof of Theorem~\ref{th:stab-fractional-gen}]
		As noted in Section~\ref{sec:prelim}, any non-trivial solution to~\eqref{eq:mainprob-frac-gen-relaxed}, with~$f$ satisfying~$f(0) \geq 0$, is such that~$u>0$ in~$\Omega$.
		Moreover, thanks to the positivity of~$u$,~\eqref{eq:vnonneg-Omegalamb'} follows as in the proof of Theorem~\ref{th:sym-frac-gen}.
		
		At this point, we only need to adapt the proof of Proposition~\ref{prop:unif-stab-each-dir}. Since~$\Omega$ is of class~$C^1$ and satisfies the uniform interior sphere condition with radius~$\mathfrak{r}_\Omega>0$, the nonlinearity~$f$ satisfies~\eqref{eq:ass-f-frac} with~$f(0) \geq 0$, and~$u$ is non-trivial, Corollary~3.4 in~\cite{dptv-parallel} yields
		\begin{equation*}
			u(x) \geq C' \, \mathfrak{r}_\Omega^s \,\mathrm{dist}\left(x,\partial\Omega\right)^s \quad\text{for all } x \in \Omega,
		\end{equation*}
		where~$C'>0$ is the constant defined in~\eqref{eq:def-Cast-bounbelowu}. Hence,~\eqref{eq:est1-stab-1dir} follows with
		\begin{equation*}
			C_1 \coloneqq \frac{\left(\mathrm{diam}(\Omega) + r_{\Omega}\right)^{n+2s+2}}{c_{n,s} \, C' \, \mathfrak{r}_\Omega^s \left(n+2s\right)}.
		\end{equation*}
		From this point onward, the proof proceeds exactly as in the previous argument.
	\end{proof}
	
	
	\section{Further stability results}
	\label{sec:other-stab-result}
	
	Before the proofs, we need the following lemma, which allows us to estimate the measure of a half-tube lying in the interior of a sufficiently regular domain. This result may be of independent interest.
	
	\begin{lemma}
		\label{lem:est-tub-lip}
		Assume that~$n \geq 2$. Let~$\Omega \subseteq \R^n$ be a bounded domain with Lipschitz boundary such that~$\partial\Omega \in  W^{2,n-1}$. Then, for any~$\gamma > 0$, we have
		\begin{equation*}
			\abs*{\left\{x \in \Omega \mid \mathrm{dist}\!\left(x,\partial\Omega\right) \leq \gamma\right\}} \leq \gamma \int_{\partial\Omega} \left[1+\gamma \,\frac{H_{-}}{ n-1}\right]^{n-1} d\Haus^{n-1}.
		\end{equation*}
		Here,~$H_{-}$ denotes the negative part of the weak mean curvature of~$\partial\Omega$.
	\end{lemma}
	\begin{proof}
		By Theorem~1 in~\cite{carlos-domain}, there exist a sequence of smooth bounded domains $\left\{\Omega_m\right\}_m$ and a constant~$\mathscr{C}_\Omega>0$, depending only on the Lipschitz constant of~$\Omega$, such that~$\Omega \Subset \Omega_m$ and
		\begin{equation}
			\label{eq:conv-Haus-dist-1}
			\mathrm{dist}_{\mathcal{H}}\!\left(\partial \Omega, \partial \Omega_m\right) \leq \frac{\mathscr{C}_\Omega}{m} \quad\text{for every } m \in \N.
		\end{equation}
		Moreover, if~$R_{\Omega} \in (0,1)$ denotes the radius of graphicality of~$\partial\Omega$, there exist a finite number of local boundary charts~$\left\{\phi^i\right\}_{i=1}^N$ and~$\left\{\psi^i_m\right\}_{i=1}^N$ describing~$\partial\Omega$ and~$\partial\Omega_m$, respectively, such that for every~$i=1,\dots,N$, the maps~$\psi^i_m$ are smooth and defined on the same reference system as~$\phi^i$. Finally, we have
		\begin{equation}
			\label{eq:conv-local-chart}
			\psi^i_m \to \phi^i \quad\text{in } W^{2,n-1}\!\left(B_{R}'\right) \quad\text{as } m \to +\infty,
		\end{equation}
		where~$R \coloneqq R_\Omega-\varepsilon_0$,~$\varepsilon_0 \in (0,R_\Omega/2)$ being a fixed constant, and~$B_{R}'$ denoting the~$(n-1)$-dimensional ball centered at the origin. Note, in particular, that~\eqref{eq:conv-Haus-dist-1} follows from the estimate on page 25 of~\cite{carlos-domain}.
		
		We also observe that if~$g$ and~$\mathcal{B}$ denote the first and the second fundamental forms of~$\partial\Omega$, respectively, then in any local chart, at~$\Haus^{n-1}$-a.e.~$x\in \partial\Omega$, we have
		\begin{equation*}
			\label{eq:fund-forms}
			g^{kj} \coloneqq \delta_{kj} - \frac{1}{1+\abs*{\nabla \phi}^2} \frac{\partial \phi}{\partial x_k'} \frac{\partial \phi}{\partial x_j'} \quad\text{and}\quad \mathcal{B}_{kj} \coloneqq - \frac{1}{\sqrt{1+\abs*{\nabla \phi}^2}} \frac{\partial^2 \phi}{\partial x_k' \partial x_j'}.
		\end{equation*}
		This implies that
		\begin{equation*}
			\label{eq:estH-chart}
			\abs*{g^{kj} \mathcal{B}_{kj}} \leq \frac{1}{\sqrt{1+\abs*{\nabla \phi}^2}} \left(1+\frac{\abs*{\nabla \phi}^2}{1+\abs*{\nabla \phi}^2}\right) \abs{\nabla^2 \phi} \leq 2 \, \frac{\abs{\nabla^2 \phi}}{\sqrt{1+\abs*{\nabla \phi}^2}},
		\end{equation*}
		and analogously for~$\partial\Omega_m$. Therefore, it follows that
		\begin{align*}
			\notag
			\int_{B_R'} \,\abs*{g^{kj} \mathcal{B}_{kj}}^{n-1} \sqrt{1+\abs*{\nabla \phi}^2} \, dx' &\leq 2^{n-1} \int_{B_R'} \,\abs{\nabla^2 \phi}^{n-1} \left(1+\abs*{\nabla \phi}^2\right)^{\! \frac{2-n}{2}} \, dx' \\
			\label{eq:finit-H}
			&\leq 2^{n-1} \int_{B_R'} \,\abs{\nabla^2 \phi}^{n-1} \, dx' < + \infty.
		\end{align*}
		As a result, a covering argument shows that~$H \in L^{n-1}\!\left(\partial\Omega\right)$.
		
		From the definition of the Hausdorff distance and~\eqref{eq:conv-Haus-dist-1}, it follows that
		\begin{equation}
		\label{eq:dist-Haus-rewritten}
			\sup_{x \in \partial \Omega} \mathrm{dist}\!\left(x, \partial \Omega_m\right) \leq \frac{\mathscr{C}_\Omega}{m}.
		\end{equation}
		Moreover, for any~$x_0 \in \left\{x \in \Omega \mid \mathrm{dist}\!\left(x,\partial\Omega\right) \leq \gamma\right\}$, there exists~$y \in \partial\Omega$ such that~$\abs*{x_0-y}=\mathrm{dist}\!\left(x_0,\partial\Omega\right) \leq \gamma$. From~\eqref{eq:dist-Haus-rewritten} there exists~$z \in \partial \Omega_m$ such that
		\begin{equation*}
			\abs*{y-z} = \mathrm{dist}\!\left(y, \partial \Omega_m\right) \leq \frac{\mathscr{C}_\Omega}{m}.
		\end{equation*}
		Since~$z \in \partial \Omega_m$, we deduce that
		\begin{equation*}
			\mathrm{dist}\!\left(x_0, \partial \Omega_m\right) \leq \abs*{x_0-z} \leq \abs*{x_0-y} + \abs*{y-z} \leq \gamma + \frac{\mathscr{C}_\Omega}{m}.
		\end{equation*}
		As a consequence, we conclude that
		\begin{equation}
			\label{eq:inclusion-tubes}
			\left\{x \in \Omega \mid \mathrm{dist}\!\left(x,\partial\Omega\right) \leq \gamma\right\} \subseteq \left\{x \in \Omega_m \Bigm| \mathrm{dist}\!\left(x,\partial\Omega_m\right) \leq \gamma + \frac{\mathscr{C}_\Omega}{m} \right\}\!.
		\end{equation}
		
		Let~$H_{\partial\Omega_m}$ denote the mean curvature of~$\partial\Omega_m$, which is well defined since~$\partial\Omega_m$ is smooth for every~$m \in \N$. Since~$\Omega_m$ is smooth, we can apply Theorem~10.20 in~\cite{gray-book} to deduce that
		\begin{multline*}
			\abs*{\left\{x \in \Omega_m \Bigm| \mathrm{dist}\!\left(x,\partial\Omega_m\right) \leq \gamma + \frac{\mathscr{C}_\Omega}{m} \right\}} \\
			\leq \left(\gamma + \frac{\mathscr{C}_\Omega}{m}\right) \int_{\partial\Omega_m} \left[1+\left(\gamma + \frac{\mathscr{C}_\Omega}{m}\right)\frac{\left(H_{\partial\Omega_m}\right)_{-}}{n-1}\right]^{n-1} d\Haus^{n-1},
		\end{multline*}
		for every~$\gamma>0$. By~\eqref{eq:inclusion-tubes}, we infer that
		\begin{equation}
			\label{eq:tube-take-lim}
			\abs*{\left\{x \in \Omega \mid \mathrm{dist}\!\left(x,\partial\Omega\right) \leq \gamma\right\}} \leq \left(\gamma + \frac{\mathscr{C}_\Omega}{m}\right) \int_{\partial\Omega_m} \left[1+\left(\gamma + \frac{\mathscr{C}_\Omega}{m}\right)\frac{\left(H_{\partial\Omega_m}\right)_{-}}{n-1}\right]^{n-1} d\Haus^{n-1}.
		\end{equation}
		
		We now claim that
		\begin{equation}
			\label{eq:conv-to-prove-meancurv}
			\int_{\partial\Omega_m} \left[1+\left(\gamma + \frac{\mathscr{C}_\Omega}{m}\right)\frac{\left(H_{\partial\Omega_m}\right)_{-}}{n-1}\right]^{n-1} d\Haus^{n-1} \to \int_{\partial\Omega} \left[1+\gamma \,\frac{H_{-}}{ n-1}\right]^{n-1} d\Haus^{n-1}
		\end{equation}
		as~$m \to +\infty$. Exploiting~\eqref{eq:conv-to-prove-meancurv} and taking the limit in~\eqref{eq:tube-take-lim} yields the conclusion.
		
		To prove~\eqref{eq:conv-to-prove-meancurv}, by a covering argument and fixing the $i$-th chart, it suffices to show that
		\begin{equation}
			\label{eq:suff-to-prove-H}
			\int_{B_R'} \mathscr{A}_m \, dx' \to \int_{B_R'} \mathscr{A} \, dx' \quad\text{as } m \to +\infty,
		\end{equation}
		where
		\begin{align*}
			\mathscr{A}_m &\coloneqq \left[1+ \frac{1}{n-1} \left(\gamma + \frac{\mathscr{C}_\Omega}{m}\right)\left(g^{kj}_{(m)} \mathcal{B}_{kj}^{(m)}\right)_{\!-} \right]^{n-1} \sqrt{1+\abs*{\nabla \psi_m}^2}, \\
			\mathscr{A} &\coloneqq \left[1+\frac{\gamma}{n-1} \left(g^{kj} \mathcal{B}_{kj}\right)_{\!-}\right]^{\! n-1} \sqrt{1+\abs*{\nabla \phi}^2}.
		\end{align*}
		Using that~$\partial\Omega$ is Lipschitz, the~$L^\infty$-bound on~$\nabla \psi_m$ which follows from Theorem~1 in~\cite{carlos-domain}, and the convergence in~\eqref{eq:conv-local-chart}, we deduce the validity of~\eqref{eq:suff-to-prove-H}, thereby concluding the proof.
	\end{proof}
	
	We can now extend the previous result to domains with positive reach, therefore proving the subsequent conclusion.
	
	\begin{lemma}
		\label{lem:tub-lip-pr}
		Assume that~$n \geq 2$. Let~$\Omega \subseteq \R^n$ be a bounded domain whose closure~$\overline{\Omega}$ has positive reach~$r_{\Omega} > 0$.  Then, for any~$\gamma > 0$, we have
		\begin{equation*}
			\abs*{\left\{x \in \Omega \mid \mathrm{dist}\!\left(x,\partial\Omega\right) \leq \gamma\right\}} \leq \gamma \left[1+\gamma \,\frac{2}{r_\Omega}\right]^{n-1} \Phi_{n-1}(\overline{\Omega},\R^n).
		\end{equation*}
		Here,~$\Phi_{n-1}(\overline{\Omega},\cdot)$ denotes the~$(n-1)$-th curvature measure associated with~$\overline{\Omega}$.
	\end{lemma}
	\begin{proof}
		For~$t \in (0,r_\Omega/2)$, let~$G_t$ be defined by~\eqref{eq:defG} and denote~$\Gamma_t = \partial G_t$. Then,~$\overline{G_t}$ is a bounded Lipschitz domain with reach at least~$r_\Omega/2$ and of class~$W^{2,n-1}$ -- in fact, it is of class~$W^{2,\infty}$, see~\cite[Theorem~4.8]{federer-curv}. Moreover, for~$\gamma>0$, we clearly have
		\begin{equation*}
			\left\{x \in \Omega \mid \mathrm{dist}\!\left(x,\partial\Omega\right) \leq \gamma\right\} \subseteq \left\{x \in G_t \mid \mathrm{dist}\!\left(x,\Gamma_t\right) \leq \gamma+t \right\}\!.
		\end{equation*}
		Hence, applying Lemma~\ref{lem:est-tub-lip}, we deduce that
		\begin{align}
			\notag
			\abs*{\left\{x \in \Omega \mid \mathrm{dist}\!\left(x,\partial\Omega\right) \leq \gamma\right\}} &\leq \abs*{\left\{x \in G_t \mid \mathrm{dist}\!\left(x,\Gamma_t\right) \leq \gamma+t \right\}} \\
			\label{eq:tube-limt}
			&\leq \left(\gamma+t\right) \int_{\Gamma_t} \left[1+\left(\gamma+t\right)\frac{H_{-}}{ n-1}\right]^{n-1} d\Haus^{n-1} \\
			\notag
			&\leq \left(\gamma+t\right) \left[1+ \left(\gamma+t\right) \frac{2}{r_\Omega}\right]^{n-1} \Haus^{n-1}(\Gamma_t),
		\end{align}
		where, in the last inequality we used the exterior sphere condition  satisfied by~$\overline{G_t}$ and the resulting upper bound on~$H_{-}$.
		
		From the Steiner formula in Remark 5.8 of~\cite{federer-curv}, we have
		\begin{equation*}
			\abs*{\left\{x \in \R^n \mid \mathrm{dist}\!\left(x,G_t\right) \leq r \right\}} = \sum_{k=0}^{n} r^{n-k} \omega_{n-k} \Phi_k(\overline{G_t},\R^n) \quad \text{for } r \leq t,
		\end{equation*}
		where~$\omega_k$ is the measure of the unit ball in~$\R^k$ and~$\Phi_k(\overline{G_t},\cdot)$ is the~$k$-th curvature measure associated with~$\overline{G_t}$ defined in~\cite{federer-curv}. Moreover, from the convergence of~$\overline{G_t}$ to~$\overline{\Omega}$ in the Hausdorff distance, we also deduce that
		\begin{equation*}
			\Phi_{n-1}(\overline{G_t},\R^n) \to \Phi_{n-1}(\overline{\Omega},\R^n) \quad\text{as } t \to 0.
		\end{equation*}
		As a consequence, we conclude that
		\begin{equation*}
			\Haus^{n-1}(\Gamma_t) \to \Phi_{n-1}(\overline{\Omega},\R^n) \quad \text{as } t \to 0.
		\end{equation*}
		See also Corollary~4.33, Proposition~7.1, Theorem~7.3, and Lemma~7.5 in~\cite{rataj-book} for further details. Therefore, taking the limit in~\eqref{eq:tube-limt}, we get the desired conclusion.
	\end{proof}
	
	We are finally in a position to prove Theorem~\ref{th:no-int-sphere}. We will highlight only the necessary modifications with respect to the proof of Theorem~\ref{th:stab-fractional} in Section~\ref{sec:proof-stabil-fractional}. The proof of Theorem~\ref{th:stab-gen-nosphere} is omitted, as it proceeds analogously, relying instead on~\eqref{eq:u-bbelow-weak-2}.
	
	\begin{proof}[Proof of Theorem~\ref{th:no-int-sphere}]
		An inspection of the proof of Theorem~\ref{th:stab-fractional} reveals that the uniform interior sphere condition is required only in Proposition~\ref{prop:unif-stab-each-dir} to establish~\eqref{eq:est-u-below} and~\eqref{eq:meas-tube}.
		
		Naturally,~\eqref{eq:est-u-below} can be simply replaced with the weaker estimate~\eqref{eq:u-bbelow-weak}. Proceeding in this way, we obtain
		\begin{equation*}
			\int_{\Omega \setminus \Omega^\star} -x_1 \,\mathrm{dist}\!\left(x,\partial\Omega\right)^{2s} dx \leq \frac{\left(\mathrm{diam}(\Omega) + r_{\Omega}\right)^{n+2s+2}}{c_{n,s} \,\gamma_{n,s} \left(n+2s\right)} \left[\mathcal{N}_s u\right]_{\sGt} \eqqcolon C_1 \left[\mathcal{N}_s u\right]_{\sGt}\!,
		\end{equation*}
		which implies
		\begin{align*}
			\abs*{\left\{x \in \Omega \setminus \Omega^\star \mid -x_1 \,\mathrm{dist}\!\left(x,\partial\Omega\right)^{2s} > \gamma\right\}} \leq \frac{C_1}{\gamma} \left[\mathcal{N}_s u\right]_{\sGt}\!.
		\end{align*}
		Moreover, we also have
		\begin{align*}
			&\abs*{\left\{x \in \Omega \setminus \Omega^\star \mid -x_1 \,\mathrm{dist}\!\left(x,\partial\Omega\right)^{2s} \leq \gamma\right\}} \\
			&\quad\leq \abs*{\left\{x \in \Omega \cap H^{\star} \mid -x_1 < \gamma^{\frac{1}{1+2s}}\right\}} + \abs*{\left\{x \in \Omega \mid \mathrm{dist}\!\left(x,\partial\Omega\right) \leq \gamma^{\frac{1}{1+2s}}\right\}} \\
			&\quad\leq \mathrm{diam}(\Omega)^{n-1} \,\gamma^{\frac{1}{1+2s}} + \abs*{\left\{x \in \Omega \mid \mathrm{dist}\!\left(x,\partial\Omega\right) \leq \gamma^{\frac{1}{1+2s}}\right\}}.
		\end{align*}
		
		We assume that~$\gamma \leq 1$ and apply Lemma~\ref{lem:tub-lip-pr} to deduce that
		\begin{equation*}
			\abs*{\left\{x \in \Omega \mid \mathrm{dist}\!\left(x,\partial\Omega\right) \leq \gamma^{\frac{1}{1+2s}}\right\}} \leq \gamma^{\frac{1}{1+2s}} \left[1+\frac{2}{r_\Omega}\right]^{n-1} \Phi_{n-1}(\overline{\Omega},\R^n).
		\end{equation*}
		As a result, we get
		\begin{equation*}
			\abs*{\left\{x \in \Omega \setminus \Omega^\star \mid -x_1 \,\mathrm{dist}\!\left(x,\partial\Omega\right)^{2s} \leq \gamma\right\}} \leq C_2 \,\gamma^{\frac{1}{1+2s}}
		\end{equation*}
		with
		\begin{equation*}
			C_2 \coloneqq \mathrm{diam}(\Omega)^{n-1} + \left[1+\frac{2}{r_\Omega}\right]^{n-1} \Phi_{n-1}(\overline{\Omega},\R^n).
		\end{equation*}
		By choosing
		\begin{equation*}
			\gamma = \left(\left(1+2s\right) \frac{C_1}{C_2}\right)^{\!\frac{1+2s}{2+2s}} \left[\mathcal{N}_s u\right]_{\sGt}^{\frac{1+2s}{2+2s}} \!,
		\end{equation*}
		we conclude that
		\begin{equation*}
			\abs*{\Omega \triangle \Omega^\star} \leq C_\sharp \left[\mathcal{N}_s u\right]_{\sGt}^{\frac{1}{2+2s}} \!,
		\end{equation*}
		where
		\begin{equation*}
			C_\sharp \coloneqq 4 \left(s+1\right) C_1^{\frac{1}{2+2s}} \left(\frac{C_2}{2s+1}\right)^{\!\frac{2s+1}{2+2s}} \!.
		\end{equation*}
		Finally, since~$C_2 \geq \mathrm{diam}(\Omega)^{n-1}$, we can ensure that~$\gamma \leq 1$ if we restrict ourselves to
		\begin{equation*}
			\left[\mathcal{N}_s u\right]_{\sGt} \leq \frac{\mathrm{diam}(\Omega)^{n-1}}{\left(2s+1\right) C_1}.
		\end{equation*}
		By setting
		\begin{equation*}
			\gamma_0 \coloneqq \min\left\{\min\left\{\frac{1}{4},\frac{1}{n}\right\}\frac{\abs*{\Omega}}{C_\sharp}, \left(\frac{\mathrm{diam}(\Omega)^{n-1}}{\left(2s+1\right) C_1}\right)^{\!\frac{1}{2+2s}}\right\}
		\end{equation*}
		and
		\begin{equation*}
			C_\flat \coloneqq 4 \left(n+3\right) \frac{C_\sharp \,\mathrm{diam}(\Omega)}{\abs*{\Omega}},
		\end{equation*}
		we get the conclusion with
		\begin{equation*}
			C_{\Omega} \coloneqq \max\left\{2 C_\flat, \frac{\mathrm{diam}(\Omega)}{\gamma_0} \right\} \!.
		\end{equation*}
	\end{proof}
	
	
	\section*{Acknowledgments} 
	\noindent 
	
	M.G.\ is member of the “Gruppo Nazionale per l'Analisi Matematica, la Probabilità e le loro Applicazioni” (GNAMPA) of the “Istituto Nazionale di Alta Matematica” (INdAM, Italy), and has been partially supported by the “INdAM - GNAMPA Project”, CUP \#E5324001950001\#.
	
	Moreover, part of this work was carried out while M.G.\ was visiting the Institut f\"ur Mathematik at Goethe-Universit\"at Frankfurt, Frankfurt am Main, whose kind hospitality is gratefully acknowledged.
	
	The authors are grateful to the referee for suggesting a comparison between the present results and those in the recent works~\cite{dprs-altcaf,fz-serrin}. They also thank Efe Akbayir for suggesting a simplification of the definition of the critical position and of the proof of Lemma~\ref{lem:inclusion}.



\begin{thebibliography}{$99\,$}
		
		\bibitem{abr}
		\href{https://doi.org/10.57262/ade/1366030751}{
			A. Aftalion, J. Busca, W. Reichel,
			\emph{Approximate radial symmetry for overdetermined boundary value problems},
			Adv. Differential Equations \textbf{4} (1999), no. 6, 907--932.}
		
		\bibitem{carlos-domain}
		\href{https://doi.org/10.1007/s00526-024-02711-x}{
			C. A. Antonini,
			\emph{Smooth approximation of Lipschitz domains, weak curvatures and isocapacitary estimates},
			Calc. Var. Partial Differential Equations \textbf{63} (2024), no. 4, Paper No. 91, 34 pp.}
		
		\bibitem{cabre-weth}
		\href{https://doi.org/10.1515/crelle-2015-0117}{
			X. Cabré, M. M. Fall, J. Solà-Morales, T. Weth,
			\emph{Curves and surfaces with constant nonlocal mean curvature: meeting Alexandrov and Delaunay},
			J. Reine Angew. Math. \textbf{745} (2018), 253--280.}
		
		\bibitem{cir-pol-fractionaltorsion}
		\href{https://doi.org/10.1090/tran/8837}{
			G. Ciraolo, S. Dipierro, G. Poggesi, L. Pollastro, E. Valdinoci,
			\emph{Symmetry and quantitative stability for the parallel surface fractional torsion problem},
			Trans. Amer. Math. Soc. \textbf{376} (2023), no. 5, 3515--3540.}
		
		\bibitem{cfmnov}
		\href{https://doi.org/10.1515/crelle-2015-0088}{
			G. Ciraolo, A. Figalli, F. Maggi, M. Novaga,
			\emph{Rigidity and sharp stability estimates for hypersurfaces with constant and almost-constant nonlocal mean curvature},
			J. Reine Angew. Math. \textbf{741} (2018), 275--294.}
		
		\bibitem{cms-parallel}
		\href{https://doi.org/10.1007/s11854-016-0011-2}{
			G. Ciraolo, R. Magnanini, S. Sakaguchi,
			\emph{Solutions of elliptic equations with a level surface parallel to the boundary: stability of the radial configuration},
			J. Anal. Math. \textbf{128} (2016), 337--353.}
		
		\bibitem{cmv-holderstab}
		\href{https://doi.org/10.1007/s10231-015-0518-7}{
			G. Ciraolo, R. Magnanini, V. Vespri,
			\emph{H\"older stability for Serrin's overdetermined problem},
			Ann. Mat. Pura Appl. (4) \textbf{195} (2016), 1333--1345.}
		
		\bibitem{dalph-ball}
		\href{https://doi.org/10.4171/IFB/401}{
			J. Dalphin,
			\emph{Uniform ball property and existence of optimal shapes for a wide class of geometric functionals},
			Interfaces Free Bound. \textbf{20} (2018), no. 2, 211--260.}
		
		\bibitem{dptv-parallel}
		\href{https://doi.org/10.1016/j.matpur.2024.06.011}{
			S. Dipierro, G. Poggesi, J. Thompson, E. Valdinoci,
			\emph{Quantitative stability for overdetermined nonlocal problems with parallel surfaces and investigation of the stability exponents},
			J. Math. Pures Appl. (9) \textbf{188} (2024), 273--319.}
		
		\bibitem{dptv-serrin}
		\href{https://doi.org/10.48550/arXiv.2309.17119}{
			S. Dipierro, G. Poggesi, J. Thompson, E. Valdinoci,
			\emph{Quantitative stability for the nonlocal overdetermined Serrin problem},
			preprint, arXiv:2309.17119, 2023.}
		
		\bibitem{dptv-antisym}
		\href{https://doi.org/10.1090/tran/8984}{
			S. Dipierro, G. Poggesi, J. Thompson, E. Valdinoci,
			\emph{The role of antisymmetric functions in nonlocal equations},
			Trans. Amer. Math. Soc. \textbf{377} (2024), no. 3, 1671--1692.}
		
		\bibitem{drov}
		\href{https://doi.org/10.4171/RMI/942}{
			S. Dipierro, X. Ros-Oton, E. Valdinoci,
			\emph{Nonlocal problems with Neumann boundary conditions},
			Rev. Mat. Iberoam. \textbf{33} (2017), no. 2, 377--416.}
			
		\bibitem{dprs-altcaf}
		\href{https://doi.org/10.48550/arXiv.2601.20493}{
			J. Domingo-Pasarin, X. Ros-Oton,
			\emph{Regularity of Lipschitz free boundaries for weak solutions of Alt-Caffarelli type problems},
			preprint, arXiv:2601.20493, 2026.}
		
		\bibitem{dyda}
		\href{https://doi.org/10.2478/s13540-012-0038-8}{
			B. Dyda,
			\emph{Fractional calculus for power functions and eigenvalues of the fractional Laplacian},
			Fract. Calc. Appl. Anal. \textbf{15} (2012), no. 4, 536--555.}
		
		\bibitem{fj}
		\href{https://doi.org/10.1051/cocv/2014048}{
			M. M, Fall, S. Jarohs,
			\emph{Overdetermined problems with fractional Laplacian},
			ESAIM Control Optim. Calc. Var. \textbf{21} (2015), no. 4, 924--938.}
			
		\bibitem{fz-serrin}
		\href{https://doi.org/10.4171/JEMS/1726}{
			A. Figalli, Y. R.-Y. Zhang,
			\emph{Serrin’s overdetermined problem in rough domains},
			J. Eur. Math. Soc. (JEMS) (2025), Published online first.}
		
		\bibitem{federer-curv}
		\href{https://doi.org/10.2307/1993504}{
			H. Federer,
			\emph{Curvature measures},
			Trans. Amer. Math. Soc. \textbf{93} (1959), no. 3, 418--491.}
		
		\bibitem{frae-book}
		\href{https://doi.org/10.1017/CBO9780511569203}{
			L. E. Fraenkel,
			\emph{An introduction to maximum principles and symmetry in elliptic problems},
			Cambridge Tracts in Mathematics, 128, Cambridge University Press, Cambridge, 2000.}
		
		\bibitem{gray-book}
		\href{https://doi.org/10.1007/978-3-0348-7966-8}{
			A. Gray,
			\emph{Tubes},
			Second edition, Progress in Mathematics, 221, Birkh\"auser Verlag, Basel, 2004.}
		
		\bibitem{reichel-elec-cap}
		\href{https://doi.org/10.4171/ZAA/719}{
			W. Reichel,
			\emph{Radial symmetry for an electrostatic, a capillarity and some fully nonlinear overdetermined problems on exterior domains},
			Z. Anal. Anwendungen \textbf{15} (1996), no. 15, 619--635.}
		
		\bibitem{reichel-ext}
		\href{https://doi.org/10.1007/s002050050034}{
			W. Reichel,
			\emph{Radial symmetry for elliptic boundary-value problems on exterior domains},
			Arch. Rational Mech. Anal. \textbf{137} (1997), no. 4, 381--394.}
		
		\bibitem{rataj-posreach}
		\href{https://doi.org/10.1002/mana.201600237}{
			J. Rataj, L. Zajíček,
			\emph{On the structure of sets with positive reach},
			Math. Nachr. \textbf{290} (2017), no. 11--12, 1806--1829.}
		
		\bibitem{rataj-book}
		\href{https://doi.org/10.1007/978-3-030-18183-3}{
			J. Rataj, M. Z\"ahle,
			\emph{Curvature measures of singular sets},
			Springer Monographs in Mathematics, Springer, Cham, 2019.}
		
		\bibitem{ro-serra}
		\href{https://doi.org/10.1016/j.matpur.2013.06.003}{
			X. Ros-Oton, J. Serra,
			\emph{The Dirichlet problem for the fractional Laplacian: regularity up to the boundary},
			J. Math. Pures Appl. (9) \textbf{101} (2014), no. 3, 275--302.}
		
		\bibitem{serrin-ARMA}
		\href{https://doi.org/10.1007/BF00250468}{
			J. Serrin,
			\emph{A symmetry problem in potential theory},
			Arch. Rational Mech. Anal. \textbf{43} (1971), no. 4, 304--318.}
		
		\bibitem{sha-divserrin}
		\href{https://doi.org/10.1080/17476933.2010.504848}{
			H. Shahgholian,
			\emph{Diversifications of Serrin's and related symmetry problems},
			Complex Var. Elliptic Equ. \textbf{57} (2012), no. 6, 653--665.}
		
		\bibitem{soave-val}
		\href{https://doi.org/10.1007/s11854-018-0067-2}{
			N. Soave, E. Valdinoci,
			\emph{Overdetermined problems for the fractional Laplacian in exterior and annular sets},
			J. Anal. Math. \textbf{137} (2019), no. 1, 101--134.}
		
		\bibitem{tamanini}
		\href{http://www.numdam.org/item?id=RSMUP_1976__56__169_0}{
			I. Tamanini,
			\emph{Il problema della capillarità su domini non regolari},
			Rend. Sem. Mat. Univ. Padova \textbf{56} (1976), 169--191.}
			
		\bibitem{weinb-ARMA}
		\href{https://doi.org/10.1007/BF00250469}{
			H. F. Weinberger, 
				\textit{Remark on the preceding paper of Serrin},
				Arch. Rational Mech. Anal. \textbf{43} (1971), no. 4, 319--320.}
		
		\bibitem{yy-eigenv}
		\href{https://doi.org/10.1142/S0219199712500484}{
			S. Yildirim Yolcu, T. Yolcu,
			\emph{Estimates for the sums of eigenvalues of the fractional Laplacian on a bounded domain},
			Commun. Contemp. Math. \textbf{15} (2013), no. 3, Paper No. 1250048, 15 pp.}
	\end{thebibliography}
\end{document}